\documentclass[12pt,reqno]{amsart}
\usepackage{graphicx}
\usepackage{amssymb,amsmath}
\usepackage{amsthm}
\usepackage[pdf]{pstricks}
\usepackage{color,graphicx}
\usepackage{hyperref}
\usepackage{color}
\usepackage{epstopdf}
\usepackage{caption}
\usepackage{subcaption}
\usepackage[dvips]{epsfig}

\numberwithin{equation}{section}
\numberwithin{figure}{section}

\newtheorem{theorem}{Theorem}[section]

\newtheorem{remark}[theorem]{Remark}
\newtheorem{lemma}[theorem]{Lemma}

\newenvironment{proof1}
{\begin{trivlist} \item[]{\em Proof }}
	{\hspace*{\fill}$\Box$\end{trivlist}}

\DeclareMathOperator{\R}{\mathbb{R}}

\begin{document}
		
\title[Modular Burgers equation]{\bf Asymptotic stability of viscous shocks \\ in the modular Burgers equation}

\author{Uyen Le}
\address[U. Le]{Department of Mathematics and Statistics, McMaster University,
Hamilton, Ontario, Canada, L8S 4K1}
\email{leu@mcmaster.ca}

%\author{Silvere Paul Nuiro}
%\address[S.P. Nuiro]{LAMIA, Universite des Antilles, Campus de Fouillole, F-97157 Pointe-a-Pitre, Guadeloupe}\email{paul.nuiro@icloud.com}

\author{Dmitry E. Pelinovsky}
\address[D. Pelinovsky]{Department of Mathematics and Statistics, McMaster University,
	Hamilton, Ontario, Canada, L8S 4K1}
\email{dmpeli@math.mcmaster.ca}

\author{Pascal Poullet}
\address[P. Poullet]{LAMIA, Universite des Antilles, Campus de Fouillole, F-97157 Pointe-a-Pitre, Guadeloupe}
\email{Pascal.Poullet@univ-antilles.fr}


\begin{abstract}
Dynamics of viscous shocks is considered in the modular Burgers equation, where the time evolution becomes complicated due to singularities produced by the modular nonlinearity. We prove that the viscous shocks are asymptotically stable under odd and general perturbations. For the odd perturbations, the proof relies on the reduction of the modular Burgers equation to a linear diffusion equation on a half-line. For the general perturbations, the proof is developed by converting the time-evolution problem to a system of linear equations coupled with a nonlinear equation for the interface position. Exponential weights in space are imposed on the initial data of general perturbations in order to gain the asymptotic decay of perturbations in time. 
We give numerical illustrations of asymptotic stability of the viscous shocks under general perturbations.
\end{abstract}

\date{\today}
\maketitle

% Applicable Analysis, Dynamics of PDEs, Nonlinearity, Studies in Applied Mathematics

\section{Introduction}

Modular nonlinearity is commonly used for approximations of nonlinear interactions between particles 
by piecewise linear functions \cite{Hedberg1,Vainstein1}. Unidirectional propagation of waves in chains of particles
is described by simplified nonlinear evolution equations with modular nonlinearity such as
the modular Burgers \cite{Radostin1,Rudenko1,Rudenko3,Hedberg3}
and modular Korteweg--de Vries \cite{Radostin2,Rudenko2,Hedberg2} equations.

Traveling solutions of modular evolution equations such as 
viscous shocks and solitary waves are found from differential equations by matching solutions of linear equations
with suitable condition at the interface where the modular nonlinearity jumps.
On the other hand, the time evolution of the modular equations is a more complicated problem
because the transport term tends to break the solution along the characteristic lines 
whereas the diffusion or dispersion terms smoothen out the solution and affect propagation of waves near the interface.
It is unclear without detailed analysis if the initial-value problem can be solved in a suitable function space
due to singularities arising from the modular nonlinearity.
Because of these reasons, stability of propagation
of traveling waves remains open.

Similar questions arise in the context of granular chains and involve 
the logarithmic versions of the Burgers and Korteweg--de Vries equations \cite{JP14,J20}. The logarithmic nonlinearity is more singular than 
the modular nonlinearity, hence questions of well-posedness and stability 
of nonlinear waves remain open for some time \cite{CP14,Natali}.

{\em The purpose of this work is to clarify stability of viscous shocks in the modular Burgers equation.} We take the modular Burgers equation in the following normalized form:
\begin{equation}
\label{Burgers}
\frac{\partial w}{\partial t} = \frac{\partial |w|}{\partial x} + \frac{\partial^2 w}{\partial x^2}, 
\end{equation}
where $w(t,x) : \mathbb{R}_+ \times \mathbb{R} \mapsto \mathbb{R}$. 
Traveling wave solutions and preliminary numerical approximations of time-dependent solutions to the modular Burgers equation
(\ref{Burgers}) were constructed with the Fourier sine series in \cite{Radostin1}. Similar results were discussed in \cite{Rudenko1,Rudenko3}.
Collisions of compactly supported pulses were considered in \cite{Hedberg1} by using heuristic
approximation methods. However, no rigorous analysis of well-posedness or numerical approximations with the control of error terms has been developed so far for the modular Burgers equation (\ref{Burgers}).

In a similar context of the diffusion equation with the piecewisely defined nonlinearity, we mention the Kolmogorov--Petrovskii--Piskunov (KPP) model with the cutoff reaction rate proposed in \cite{KPP1}. Asymptotic stability of viscous shocks (stationary fronts) was analyzed in \cite{KPP2} and more recently in \cite{KPP3,KPP4}. 

Viscous shocks and metastable $N$-waves of the classical Burgers equation were studied in \cite{B1} and more recently in \cite{B2,B3}. Stability arguments for viscous shocks and metastable $N$-waves can be developed by using the linearization analysis and dynamical system methods. Viscous shocks are also useful for analysis of the enstrophy growth in the limit of small dissipation, see \cite{D1,D2} and references therein. 

Non-smoothness of the nonlinear term in the modular Burgers equation (\ref{Burgers}) restricts us from using the dynamical system methods in the analysis of asymptotic stability of viscous shocks. Nevertheless, we are able to use the linearized estimates due to the piecewise definition of the nonlinear term in this model. 

{\em The main novelty of this paper is the rigorous analysis of the modular nonlinearity.} We keep the functional-analytic framework as simple as possible. If the perturbation has the odd spatial symmetry, the asymptotic stability result follows from analysis of the linear diffusion equation. 
For general perturbations, we impose the spatial exponential decay on the initial data in order to gain the asymptotic decay of perturbations in time. This technique is definitely not novel, see 
\cite{Sch,Hilder,Sat} for earlier studies in a similar context. Further improvements of the asymptotic stability results in less restrictive function spaces are left for future work.

The paper is organized as follows. Main results are described in Section \ref{sec-main}. Properties of solutions of the linear diffusion and Abel integral equations are reviewed in Section \ref{sec-preliminary}. 
Asymptotic stability of viscous shocks in the space of odd and general functions is proven in Sections \ref{section-odd} and \ref{section-general} respectively. Numerical illustrations are given in Section \ref{section-numerics}. The summary and open directions are described 
in Section \ref{section-conclusion}.

\section{Main results}
\label{sec-main}

In what follows, we use the classical notations $H^k(\mathbb{R})$ for the Sobolev space of squared integrable distributions on $\mathbb{R}$ with squared 
integrable derivatives up to the integer order $k \in \mathbb{N}$. 
In particular, the norms in $H^1$ and $H^2$ are defined by 
\begin{eqnarray*}
\| f \|_{H^1} & := & \left( \| f \|^2_{L^2} + \| f' \|_{L^2}^2 \right)^{1/2}, \\
\| f \|_{H^2} & := & \left( \| f \|^2_{L^2} + \| f' \|_{L^2}^2 + \| f'' \|_{L^2}^2 \right)^{1/2},
\end{eqnarray*}
Similarly, we consider $W^{1,\infty}$ and $W^{2,\infty}$ for bounded functions with bounded derivatives up the first and second order respectively. To simplify the notations, we use 
$$
\| f \|_{H^k \cap W^{k,\infty}} := \max\{ \|f \|_{H^k}, \| f \|_{W^{k,\infty}}\}.
$$
By Sobolev's embedding, if $f \in H^2(\mathbb{R})$, then $f \in C^1(\mathbb{R}) \cap W^{1,\infty}(\mathbb{R})$ and $f$ and $f'$ decay to zero at infinity. In many cases throughout our work, if $f \in W^{2,\infty}(\mathbb{R})$, then $f$ will be considered in the class of functions with piecewise continuous $f''$.

Basic properties of the heat kernel, convolution estimates, solutions to the linear diffusion equations, and solutions to the Abel integral equations are reviewed in Section \ref{sec-preliminary}.

{\em The traveling viscous shock} of the modular Burgers equation (\ref{Burgers}) can be found in the closed analytical form. Substituting $w(t,x) = W_c(x-ct)$ in (\ref{Burgers}) yields the differential equation
\begin{equation}
\label{ode}
W_c''(x) + {\rm sign}(W_c) W_c'(x) + c W_c'(x) = 0.
\end{equation}
Solutions of (\ref{ode}) are piecewise $C^2$ functions satisfying the interface condition
\begin{equation}
\label{interface}
[W_c'']^+_-(x_0) = -2 W_c'(x_0)
\end{equation}
at each interface located at $x_0$, where $[f]^+_-(x_0) = f(x_0^+) - f(x_0^-)$ is the jump of a piecewise
continuous function $f$ across $x_0$. Assuming a single interface at $x_0 = 0$ and the boundary conditions
$W_c(x) \to W_{\pm}$ as $x \to \pm \infty$ with $W_- < 0 < W_+$, we obtain the exact solution
to the differential equation (\ref{ode}) satisfying the jump condition (\ref{interface}) in the form
\begin{equation}
\label{shock}
W_c(x) = \left \{ \begin{array}{ll} W_+ (1 - e^{-(1+c)x}), \quad & x > 0, \\
W_- (1 - e^{(1-c)x}), \quad & x < 0, \end{array} \right.
\end{equation}
with the uniquely defined speed
\begin{equation}
\label{speed}
c = \frac{W_+ + W_-}{W_- - W_+}.
\end{equation}

If $W_+ = -W_-$, then $c = 0$ and the viscous shock $W_0$ is time-independent.
Moreover, the modular Burgers equation (\ref{Burgers}) on the line $\mathbb{R}$
is closed on the half-line in the space of odd functions. In this case, the evolution equation with the normalized boundary condition
$W_+ \equiv 1$ takes the form:
\begin{equation}
\label{Burgers-half}
\left\{ \begin{array}{l}
w_t = w_x + w_{xx}, \quad
x > 0, \\
w(t,0) = 0, \\
w(t,x) \to 1 \quad \mbox{\rm as} \;\; x \to +\infty, \end{array} \right.
\end{equation}
subject to the positivity condition
\begin{equation}
\label{w-positivity}
w(t,x) > 0, \quad x > 0.
\end{equation}
The classical solution of the boundary-value problem
(\ref{Burgers-half}) satisfies the constraint
\begin{equation}
\label{interface-half}
w_x(t,0^+) + w_{xx}(t,0^+) = 0.
\end{equation}
If a classical solution $w(t,x) : \mathbb{R}_+ \times \mathbb{R}_+ \mapsto \mathbb{R}$ to the boundary-value problem (\ref{Burgers-half}) is extended to the odd function 
$w_{\rm ext}(t,x) : \mathbb{R}_+ \times \mathbb{R}\mapsto \mathbb{R}$, then $w_{\rm ext}(t,\cdot)$ is a piecewise $C^2$ function satisfying the interface condition
\begin{equation}
\label{interface-half-extended}
[w_{xx}]^+_-(t,0) = -2 w_x(t,0),
\end{equation}
where $w \equiv w_{\rm ext}$ for simplicity of notations.

The following theorem states the asymptotic stability of the  viscous shock (\ref{shock})
with $c = 0$ under the odd perturbations from the analysis of the
boundary-value problem (\ref{Burgers-half}) subject to the positivity condition
(\ref{w-positivity}) and the boundary constraint (\ref{interface-half}).
The proof of this theorem is presented in Section \ref{section-odd}.

\begin{theorem}
\label{theorem-1}
For every $\epsilon > 0$ there is $\delta > 0$ such that every odd $w_0$ satisfying
\begin{equation}
\label{initial-1}
\| w_0 - W_0 \|_{H^2} < \delta
\end{equation}
generates the unique odd solution $w(t,x)$ to the modular Burgers equation (\ref{Burgers}) with $w(0,x) = w_0(x)$ satisfying
\begin{equation}
\label{final-1}
\| w(t,\cdot) - W_0 \|_{H^2} < \epsilon, \quad t > 0
\end{equation}
and
\begin{equation}
\label{scattering-1}
\| w(t,\cdot) - W_0 \|_{W^{2,\infty}} \to 0 \quad \mbox{\rm as} \quad t \to + \infty.
\end{equation}
The solution belongs to the class of functions such that $w - W_0 \in C(\mathbb{R}_+,H^2(\mathbb{R}))$.
\end{theorem}

\begin{remark}
Since $H^2(\mathbb{R})$ is continuously embedded into $C^1(\mathbb{R}) \cap W^{1,\infty}(\mathbb{R})$ with functions and their first derivatives decaying to zero at infinity, whereas $W_0(0) = 0$, $W_0'(0) = 1$, and $W_0(x) \to 1$ as $x \to \infty$,
the only interface of the solution $w(t,\cdot)$ 
in Theorem \ref{theorem-1} with small $\epsilon > 0$ 
is located at the origin. The positivity condition (\ref{w-positivity}) is satisfied for all $t \in \mathbb{R}_+$.
\end{remark}

\begin{remark}
The following transformation
\begin{equation}
\label{transformation}
w(t,x) = \left\{ \begin{array}{ll} W_+ v((1+c)^2 t, (1+c) (x-ct)), \quad & x - ct > 0, \\
W_- v((1-c)^2 t, (1-c)(x-ct), \quad & x - ct < 0, \end{array} \right.
\end{equation}
where $c$ is given by (\ref{speed}), relates solutions
$w(t,x)$ with $W_+ \neq -W_-$ to solutions $v(t,x)$ with normalized boundary conditions $v(t,x) \to \pm 1$ as $x \to \pm \infty$.
If $v(t,x)$ is odd in $x$, then it satisfies the same boundary-value problem (\ref{Burgers-half})
subject to the same constraints (\ref{w-positivity}) and (\ref{interface-half}).
Hence Theorem \ref{theorem-1} can be extended trivially to the traveling viscous shock $W_c$ with $c \neq 0$
under the odd perturbation of $v(t,x)$ in (\ref{transformation}).
\end{remark}

For the general perturbations, we consider the solution $w(t,x)$ to the modular
Burgers equation (\ref{Burgers}) with exactly one interface located dynamically at $x = \xi(t)$. Without loss of generality, we assume $\xi(0) = 0$. The evolution equation with the normalized boundary conditions $W_+ = -W_- \equiv 1$
takes the form:
\begin{equation}
\label{Burgers-interface}
\left\{ \begin{array}{l}
w_t = \pm w_x + w_{xx}, \quad
\pm (x -\xi(t)) > 0, \\
w(t,\xi(t)) = 0, \\
w(t,x) \to \pm 1 \quad \mbox{\rm as} \;\; x \to \pm \infty, \end{array} \right.
\end{equation}
subject to the positivity conditions
\begin{equation}
\label{w-positivity-general}
\pm w(t,x) > 0, \quad \pm (x - \xi(t)) > 0.
\end{equation}
Piecewise $C^2$ solutions of the boundary-value problem
(\ref{Burgers-interface}) satisfy the interface condition
\begin{equation}
\label{interface-interface}
[w_{xx}]^+_-(t,\xi(t)) = - 2 w_x(t,\xi(t)),
\end{equation}
whereas the boundary condition $w(t,\xi(t)) = 0$ implies
\begin{equation}
\label{interface-continuity}
w_t(t,\xi(t)) + \xi'(t) w_x(t,\xi(t)) = 0,
\end{equation}
for continuous $w_t$ and $w_x$ across the interface at $x = \xi(t)$.

The following theorem states the asymptotic stability of the  viscous shock (\ref{shock}) with $c = 0$ under general perturbations from the analysis of the
boundary-value problem (\ref{Burgers-interface}) subject to the positivity conditions
(\ref{w-positivity-general}) and the interface conditions (\ref{interface-interface}) and (\ref{interface-continuity}).
The proof of this theorem is presented in Section \ref{section-general}.

\begin{theorem}
\label{theorem-2}
Fix $\alpha \in \left(0,\frac{1}{2}\right)$. 
For every $\epsilon > 0$ there is $\delta > 0$ such that every $w_0$ satisfying
\begin{equation}
\label{initial-2}
\| w_0 - W_0 \|_{H^2 \cap W^{2,\infty}} + \| e^{\alpha |\cdot|} (w_0 - W_0) \|_{W^{2,\infty}} < \delta
\end{equation}
generates the unique solution $w(t,x)$ to the modular Burgers equation (\ref{Burgers}) with $w(0,x) = w_0(x)$ satisfying 
\begin{equation}
\label{final-2}
\| w(t,\cdot + \xi(t)) - W_0 \|_{H^2 \cap W^{2,\infty}} < \epsilon, \quad t > 0
\end{equation}
and 
\begin{equation}
\label{scattering-2}
\| w(t,\cdot + \xi(t)) - W_0 \|_{W^{2,\infty}} \to 0 \quad \mbox{\rm as} \quad t \to + \infty,
\end{equation}
where $\xi \in C^1(\mathbb{R}_+)$ is the uniquely determined interface position satisfying $\xi(0) = 0$ and $\xi' \in L^1(\mathbb{R}_+) \cap L^{\infty}(\mathbb{R}_+)$.
The solution belongs to the class of functions such that
\begin{equation}
\label{space-2}
w(t,\cdot + \xi(t)) - W_0 \in C(\mathbb{R}_+,H^2(\mathbb{R}) \cap W^{2,\infty}(\mathbb{R}))
\end{equation}
and
\begin{equation}
\label{space-2-exp}
e^{\alpha |\cdot + \xi(t)|} [w(t,\cdot + \xi(t)) - W_0] \in C(\mathbb{R}_+, W^{2,\infty}(\mathbb{R})).
\end{equation}
\end{theorem}

\begin{remark}
The additional requirement $w_0 - W_0 \in H^2(\mathbb{R}) \cap W^{2,\infty}(\mathbb{R})$ for the initial data $w_0$ in Theorem \ref{theorem-2} compared
to $w_0 - W_0 \in H^2(\mathbb{R})$ in Theorem \ref{theorem-1} is due to the necessity to control $\xi'(t)$ from the interface conditions (\ref{interface-interface}) and (\ref{interface-continuity}).
As we will show in Lemma \ref{lem-interface}, this is possible if the solution stays in the class of functions satisfying (\ref{space-2}).
\end{remark}

\begin{remark}
We assume in (\ref{initial-2}) that $|w_0(x) - W_0(x)| \to 0$ as $|x| \to \infty$ 
at least exponentially with the decay rate $\alpha \in (0,\frac{1}{2})$. This gives the asymptotic stability resulting in 
$$
\xi'(t) \to 0 \quad \mbox{\rm and} \quad \| w(t,\cdot + \xi(t)) - W_0 \|_{W^{2,\infty}} \to 0 \quad \mbox{\rm as} \quad t \to +\infty.
$$ 
The exponential decay in space is preserved in time as is shown in (\ref{space-2-exp}). 
It is opened for further studies to relax the exponential decay 
requirement on the general initial data $w_0$.
\end{remark}

\begin{remark}
Thanks to the transformation (\ref{transformation}), Theorem \ref{theorem-2} can be extended trivially 
to the traveling viscous shock $W_c$ with $c \neq 0$ under a general perturbation of $v(t,x)$.
\end{remark}

Numerical illustrations of the asymptotic stability of the  viscous shock (\ref{shock}) with $c = 0$ for two examples of general perturbations are given in Section \ref{section-numerics},
where the boundary-value problem (\ref{Burgers-interface}) with (\ref{w-positivity-general}), (\ref{interface-interface}),
and (\ref{interface-continuity}) is approximated by using the finite-difference method. Error of the finite-difference numerical approximation is controlled by the standard analysis. The two examples are constructed for perturbations with the Gaussian and exponential decay at infinity. Numerical simulations illustrate the asymptotic stability result of Theorem \ref{theorem-2}.

%The paper is concluded in Section \ref{section-conclusion}, where open directions are discussed.

\section{Preliminary results}
\label{sec-preliminary}

The heat kernel is defined by $G(t,x) := \frac{1}{\sqrt{4\pi t}} e^{-\frac{x^2}{4t}}$. It follows from explicit computations of integrals that the heat kernel satisfies the properties:
\begin{equation}
\label{heat-kernel}
\| G(t,\cdot) \|_{L^1(\mathbb{R})} = 1, \quad
\| G(t,\cdot) \|_{L^2(\mathbb{R})} = \frac{1}{(8 \pi t)^{1/4}}, \quad 
\| G(t,\cdot) \|_{L^{\infty}(\mathbb{R})} = \frac{1}{(4 \pi t)^{1/2}},
\end{equation}
\begin{equation}
\label{heat-kernel-derivative}
\| \partial_x G(t,\cdot) \|_{L^1(\mathbb{R})} 
= \frac{1}{(\pi t)^{1/2}}, \quad 
\| \partial_x G(t,\cdot) \|_{L^2(\mathbb{R})} 
	= \frac{1}{2 (8 \pi)^{1/4} t^{3/4}},  
\end{equation}
and
\begin{equation} 
\label{heat-kernel-derivative-last}
\| \partial_x G(t,\cdot) \|_{L^{\infty}(\mathbb{R})} 
= \frac{1}{2 (2 \pi e)^{1/2} t},
\end{equation}
The heat kernel is used to solve the following Dirichlet problem for the linear diffusion equation on the half-line: 
\begin{equation}
\label{Burgers-odd-tilde}
\left\{ \begin{array}{ll}
v_t = v_{xx}, \quad & x > 0, \quad t > 0, \\
v(t,0) = 0, \quad & t > 0, \\
v(0,x) = v_0(x), \quad & x > 0. \end{array} \right.
\end{equation}
For a rather general class of functions $v_0(x) : \mathbb{R}_+ \mapsto \mathbb{R}$ (not necessarily decaying to zero at
infinity), the Dirichlet problem (\ref{Burgers-odd-tilde}) can be solved
by the method of images:
\begin{eqnarray}
v(t,x) = \int_0^{\infty} v_0(y) \left[ G(t,x-y) - G(t,x+y) \right] dy.
\label{u-tilde-sol}
\end{eqnarray}
The convolution integrals in (\ref{u-tilde-sol}) are analyzed with the generalized Young's inequality:
\begin{equation}
\label{Young}
\| f \ast g \|_{L^r(\mathbb{R})} \leq \| f \|_{L^p(\mathbb{R})} \| g \|_{L^q(\mathbb{R})}, \quad p,q,r \geq 1, \quad 1 + \frac{1}{r} = \frac{1}{p} + \frac{1}{q},
\end{equation}
for every $f \in L^p(\mathbb{R})$ and $g \in L^q(\mathbb{R})$, 
where $(f \ast g)(x) := \int_{\mathbb{R}} f(y) g(x-y) dy$ is the convolution integral. When integration is needed to be restricted on $\mathbb{R}_+$ as in (\ref{u-tilde-sol}), we can use the characteristic function $\chi_{\mathbb{R}_+}$ defined by 
$\chi_{\mathbb{R}_+}(x) = 1$ for $x > 0$ and $\chi_{\mathbb{R}_+}(x) = 0$ for $x < 0$.

For the inhomogeneous linear diffusion equation on the half-line: 
\begin{equation}
\label{Burgers-odd-tilde-inhomog}
\left\{ \begin{array}{ll}
v_t = v_{xx} + f(t,x), \quad & x > 0, \quad t > 0, \\
v(t,0) = 0, \quad & t > 0, \\
v(0,x) = v_0(x), \quad & x > 0. \end{array} \right.
\end{equation}
with given $v_0(x) : \mathbb{R}_+ \mapsto \mathbb{R}$ 
and $f(t,x) : \mathbb{R}_+ \times \mathbb{R}_+ \mapsto \mathbb{R}$, 
the exact solution is written in the form
\begin{eqnarray}
\nonumber
v(t,x) & = & \int_0^{\infty} v_0(y) \left[ G(t,x-y) - G(t,x+y) \right] dy \\
&& 
+ \int_0^t \int_0^{\infty} f(\tau,y) \left[ G(t-\tau,x-y) - G(t-\tau,x+y) \right] dy d \tau.
\label{u-tilde-sol-inhomog}
\end{eqnarray}

Next, we analyze the following initial-value problem:
\begin{equation}
\label{eta-problem}
\left\{ \begin{array}{ll}
\nu_t = \nu_y + \nu_{yy} + 2 \gamma(t) \delta(y), \quad & y \in \mathbb{R}, \;\; t > 0, \\
\nu(0,y) = 0, \quad & y \in \mathbb{R}, \end{array} \right.
\end{equation}
where $\delta$ is the Dirac distribution centered at zero and $\gamma \in C(\mathbb{R}_+)$ is a given function. In order to construct the exact solution to this problem, 
we use the Laplace transform in time $t$ defined by
\begin{equation}
\label{laplace-transform}
\hat{\gamma}(p) := \mathcal{L}(\gamma)(p) = \int_0^{\infty} \gamma(t) e^{-p t} dt, \quad p \geq 0.
\end{equation}
We also use the following relations from the table of Laplace transforms for every $y \in \mathbb{R}$:
\begin{equation}
\label{table-integral}
\mathcal{L} \left( \frac{1}{\sqrt{\pi t}} e^{-\frac{y^2}{4t}} \right) = 
\frac{1}{\sqrt{p}} e^{-\sqrt{p} |y|}, \quad p > 0
\end{equation}
and
\begin{equation}
\label{table-integral-derivative}
\mathcal{L} \left( \frac{1}{\sqrt{\pi t}} \frac{y}{2t} e^{-\frac{y^2}{4t}} \right) = {\rm sign}(y) \; e^{-\sqrt{p} |y|}, \quad p > 0.
\end{equation}
The following lemma gives the exact solution to the initial-value problem 
(\ref{eta-problem}).

\begin{lemma}
	\label{prop-tech}
For every $\gamma \in C(\mathbb{R}_+)$, there exists the unique 
solution to the initial-value problem (\ref{eta-problem}) in the exact form:
	\begin{equation}
	\label{def-nu}
	\nu(t,y) := 2 \int_0^t \frac{\gamma(\tau)}{\sqrt{4 \pi (t-\tau)}}  
	e^{-\frac{(y + t - \tau)^2}{4 (t-\tau)}} d \tau, \quad y \in \mathbb{R}, 
	\quad t > 0.
	\end{equation}
Moreover, $\nu$ belongs to the class of functions in $C(\mathbb{R}_+,H^1(\mathbb{R}) \cap W^{1,\infty}(\mathbb{R}_+))$ satisfying
\begin{equation}
\label{bc-nu}
\nu_y(t,0^{\pm}) + \frac{1}{2} \nu(t,0) = \mp \gamma(t), \quad t > 0.
\end{equation}
\end{lemma}

\begin{proof}
By using (\ref{laplace-transform}) and (\ref{table-integral}), we compute from (\ref{def-nu}):
	$$
	\hat{\nu}(p,y) = \mathcal{L}\left(\frac{1}{\sqrt{\pi t}}  e^{-\frac{(y+t)^2}{4 t}} \right)(p) \times \mathcal{L}(\gamma)(p)
	= e^{-\frac{y}{2}} \frac{e^{-\sqrt{p + \frac{1}{4}} |y|}}{\sqrt{p + \frac{1}{4}}} \hat{\gamma}(p),
	$$
where we have used properties of the Laplace transform, e.g. 
$$
\mathcal{L}(f(t) e^{-\frac{t}{4}})(p) = \hat{f}(p + \frac{1}{4})
\quad \mbox{\rm and} \quad 
\mathcal{L}\left( \int_0^t f(\tau) g(t-\tau) d\tau \right)(p) = \hat{f}(p) \hat{g}(p).
$$
Differentiations of $\hat{\nu}(p,y)$ in $y$ yield 
	\begin{eqnarray}
	\hat{\nu}_y & = & - \frac{1}{2} \hat{\nu} - {\rm sign}(y) e^{-\frac{y}{2}} e^{-\sqrt{p+\frac{1}{4}}|y|}\hat{\gamma}(p), \label{der-xi-1}\\
	\hat{\nu}_{yy} & = & -\frac{1}{2} \hat{\nu}_y - 2 \delta(y) \hat{\gamma}(p) + \frac{1}{2} {\rm sign}(y) e^{-\frac{y}{2}}  e^{-\sqrt{p+\frac{1}{4}} |y|} \hat{\gamma}(p)
	+ \left( p + \frac{1}{4} \right) \hat{\nu}. 	\label{der-xi-2} 
	\end{eqnarray}
Combining (\ref{der-xi-1}) and (\ref{der-xi-2}) yields
\begin{equation}
\hat{\nu}_{yy} = - \hat{\nu}_y - 2 \delta(y) \hat{\gamma}(p) + p \hat{\nu},
\end{equation}
which becomes the initial-value problem (\ref{eta-problem})
after the inverse Laplace transform. It follows from (\ref{der-xi-1}) for $p \geq 0$ that 
$$
\hat{\nu}_y(p,0^{\pm}) = -\frac{1}{2} \hat{\nu}(p,0) \mp \hat{\gamma}(p),
$$
which yields (\ref{bc-nu}) after the inverse Laplace transform. Uniqueness of the solution (\ref{def-nu}) is proven from uniqueness of the zero solution in the homogeneous version of the initial-value problem (\ref{eta-problem}).

Next, we estimate the solution (\ref{def-nu}) in Sobolev spaces 
provided that $\gamma \in C(\mathbb{R}_+)$. By using (\ref{heat-kernel}), we obtain
	\begin{eqnarray}
	\label{nu-1}
		\| \nu(t,\cdot) \|_{L^2(\mathbb{R})} \leq \frac{2}{(8 \pi)^{1/4}}
		\int_0^t \frac{|\gamma(\tau)|}{(t - \tau)^{1/4}} d \tau
	\end{eqnarray}
and
	\begin{eqnarray}
	\label{nu-2}
	\| \nu(t,\cdot) \|_{L^{\infty}(\mathbb{R})} \leq 
	\frac{1}{\sqrt{\pi}}
	\int_0^t \frac{|\gamma(\tau)|}{(t - \tau)^{1/2}} d \tau.
\end{eqnarray}
The derivative $\nu(t,y)$ in $y$ is given by
	\begin{eqnarray}
		\nu_y(t,y) = -\int_0^t \frac{\gamma(\tau) (y+t-\tau)}{\sqrt{4\pi(t-\tau)^3}} e^{-\frac{(y+t-\tau)^2}{4(t-\tau)}} d \tau.
	\end{eqnarray}
By using (\ref{heat-kernel-derivative}), we obtain 
	\begin{eqnarray}
	\label{nu-3}
		\| \nu_y(t,\cdot) \|_{L^2(\mathbb{R})} & \leq & \frac{1}{(8 \pi)^{1/4}}
		\int_0^t \frac{|\gamma(\tau)|}{(t - \tau)^{3/4}} d \tau,
	\end{eqnarray}
hence $\nu \in C(\mathbb{R}_+,H^1(\mathbb{R}))$ if $\gamma \in C(\mathbb{R}_+)$. By Sobolev embedding, $\nu \in C(\mathbb{R}_+,L^{\infty}(\mathbb{R}))$, which also follows from (\ref{nu-2}). 

It remains to show that $\nu_y \in	C(\mathbb{R}_+,L^{\infty}(\mathbb{R}_+))$. Due to (\ref{heat-kernel-derivative-last}), estimates on $\| \nu_y(t,\cdot) \|_{L^{\infty}(\mathbb{R})}$ which are similar to (\ref{nu-2}) produce a non-integrable singularity in the convolution integral in time. Nevertheless, we show 
hereafter that 
$\| \nu_y(t,\cdot) \|_{L^{\infty}(\mathbb{R}_+)}$ can be estimated 
in terms of $|\gamma(t)|$.
	
The initial-value problem (\ref{eta-problem}) can be rewritten in the piecewise form:
	\begin{equation}
	\label{eta-problem-plus-minus}
	\left\{ \begin{array}{ll}
	\nu_t = \nu_y + \nu_{yy}, \quad & \pm y > 0, \quad t > 0, \\
	\nu_y(t,0^+) - \nu_y(t,0^-) = -2 \gamma(t), \quad & t > 0, \\
	\nu(0,y) = 0, & y \in \mathbb{R}. \end{array} \right.
	\end{equation}
	With the transformation 
	$$
	\nu(t,y) = e^{-\frac{y}{2} - \frac{t}{4}} \tilde{\nu}(t,y),
	$$
the initial-boundary-value problem (\ref{eta-problem-plus-minus}) is equivalently written as 
	\begin{equation}
	\label{eta-problem-tilde}
	\left\{ \begin{array}{ll}
	\tilde{\nu}_t = \tilde{\nu}_{yy}, \quad & \pm y > 0, \quad t > 0, \\
	\tilde{\nu}_y(t,0^+) - \tilde{\nu}_y(t,0^-) = -2 \gamma(t) e^{\frac{t}{4}}, \quad & t > 0, \\
	\tilde{\nu}(0,y) = 0, & y \in \mathbb{R}. \end{array} \right.
	\end{equation}
	Due to the parity symmetry of the boundary and initial conditions in (\ref{eta-problem-tilde}), $\tilde{\nu}$ is even in $y$, 
	$\tilde{\nu}_y$ is odd in $y$, so that 
	$\tilde{\nu}_y$ solves Dirichlet's problems for the diffusion 
	equation on the quarter planes $\{ y > 0, \;\; t > 0 \}$ 
	and $\{ y < 0, \;\; t > 0 \}$ subject to the boundary conditions $\tilde{\nu}_y(t,0^+) = - \gamma(t) e^{\frac{t}{4}}$ and 
	$\tilde{\nu}_y(t,0^-) = \gamma(t) e^{\frac{t}{4}}$ respectively. It follows by the maximum principle that
	\begin{equation}
	\label{estimates-half-whole}
	\| \tilde{\nu}_y(t,\cdot) \|_{L^{\infty}(\mathbb{R})} \leq |\gamma(t)| e^{\frac{t}{4}}, \quad t > 0,
		\end{equation}
	which yields 
\begin{equation}
\label{nu-4}
	\| \nu_y(t,\cdot) \|_{L^{\infty}(\mathbb{R}_+)} \leq \frac{1}{2} \| \nu(t,\cdot) \|_{L^{\infty}(\mathbb{R}_+)} + |\gamma(t)|, \quad t > 0,
\end{equation}
since $\nu_y + \frac{1}{2} \nu = e^{-\frac{y}{2} - \frac{t}{4}} \tilde{\nu}_y$ and $e^{-\frac{y}{2}} \leq 1$ for $y \geq 0$. Hence, $\nu_y \in	C(\mathbb{R}_+,L^{\infty}(\mathbb{R}_+))$.
\end{proof}

\begin{remark}
	Since $e^{-\frac{y}{2}}$ is unbounded for $y \in \mathbb{R}_-$, no bound on 
	$\| \nu_y(t,\cdot) \|_{L^{\infty}(\mathbb{R}_-)}$ can be obtained from the estimate (\ref{estimates-half-whole}). However, we only need to use $\nu(t,y)$ for $t > 0$ and $y > 0$.
\end{remark}

Next, we consider inverting the linear equation 
\begin{eqnarray}
\label{eq-Abel}
\mathcal{M}(\gamma) = \frac{1}{\sqrt{4 \pi t}} \int_0^{\infty} 
f(\eta) e^{-\frac{\eta^2}{4t}} d \eta, \quad t > 0,
\end{eqnarray}
where
\begin{eqnarray}
\label{operator-M}
\mathcal{M}(\gamma) := \int_0^t \frac{\gamma(\tau)}{\sqrt{\pi (t-\tau)}}  d\tau 
-\int_0^t \frac{\gamma(\tau)}{\sqrt{4 \pi (t-\tau)}}
	\int_0^{\infty} e^{-\frac{\eta}{2}} e^{-\frac{\eta^2}{4(t - \tau)}} d\eta  d \tau
\end{eqnarray}
and $f \in W^{1,\infty}(\mathbb{R}_+)$ is a given function. 
The invertion problem (\ref{eq-Abel}) is related to 
Abel's integral equation \cite{Tamarkin,Tonelli}. We use again the Laplace transform in time $t$, 
as is defined in (\ref{laplace-transform}). The following lemma gives the exact solution to the integral equation (\ref{eq-Abel}) in the space of bounded functions. 

\begin{lemma}
	\label{lem-Abel}
	For every $f \in W^{1,\infty}(\mathbb{R}_+)$ satisfying $f(0) = 0$, there exists the unique solution $\gamma \in L^{\infty}(\mathbb{R}_+)$ 
	to the integral equation (\ref{eq-Abel}) in the exact form:
	\begin{equation}
	\label{sol-Abel}
	\gamma(t) = \frac{1}{\sqrt{4\pi t}} \int_0^{\infty} 
	f(\eta) \left( \frac{\eta + t}{2t} \right) e^{-\frac{\eta^2}{4t}} d \eta, \quad t > 0,
	\end{equation}
	or, equivalently, 
	\begin{equation}
	\label{sol-Abel-equivalent}
	\gamma(t) = \frac{1}{\sqrt{4\pi t}} \int_0^{\infty} 
\left[ 	f'(\eta) + \frac{1}{2} 	f(\eta) \right] e^{-\frac{\eta^2}{4t}} d \eta, \quad t > 0,
	\end{equation}
\end{lemma}

\begin{proof}
	By using (\ref{laplace-transform}) and (\ref{table-integral}), we rewrite the integral equation (\ref{eq-Abel}) in the product form:
	$$
	\frac{1}{\sqrt{p}} \hat{\gamma}(p) - \frac{1}{2 \sqrt{p}} \hat{\gamma}(p) \int_0^{\infty} e^{-\frac{\eta}{2}} e^{-\sqrt{p} \eta} d \eta = \frac{1}{2 \sqrt{p}} \int_0^{\infty} f(\eta) e^{-\sqrt{p} \eta} d\eta, \quad p > 0.
	$$
	Evaluating the integral gives the solution in the Laplace transform space:
	$$
	\hat{\gamma}(p) = \frac{1}{2} \int_0^{\infty} f(\eta) \left(1 + \frac{1}{2 \sqrt{p}} \right) e^{-\sqrt{p} \eta} d\eta.
	$$
	After the inverse Laplace transform, we obtain 
	the exact solution (\ref{sol-Abel}) with the use of (\ref{table-integral-derivative}). The equivalent form (\ref{sol-Abel-equivalent}) is obtained from (\ref{sol-Abel}) after 
	integration by parts if $f \in W^{1,\infty}(\mathbb{R}_+)$ and $f(0) = 0$.
	It follows from (\ref{sol-Abel-equivalent}) that 
	$$
	\sup_{t \geq 0} |\gamma(t)| \leq \frac{1}{2} \| f' \|_{L^{\infty}(\mathbb{R}_+)} + \frac{1}{4} \| f \|_{L^{\infty}(\mathbb{R}_+)},
	$$
so that $\gamma \in L^{\infty}(\mathbb{R}^+)$.
\end{proof}

Similarly to Lemma \ref{lem-Abel}, we consider inverting of the linear equations 
\begin{equation}
\label{eq-Abel-integral}
\mathcal{M}(\gamma) = \int_0^t \frac{1}{\sqrt{4 \pi (t-\tau)}} \int_0^{\infty} 
g(\tau,\eta) e^{-\frac{\eta^2}{4(t-\tau)}} d \eta d \tau, \quad t > 0
\end{equation}
and
\begin{equation}
\label{eq-Abel-integral-local}
\mathcal{M}(\gamma) = \int_0^t \frac{h(\tau)  d\tau}{\sqrt{4 \pi (t-\tau)}}, \quad t > 0
\end{equation}
where $\mathcal{M}(\gamma)$ is given by (\ref{operator-M}),
$g \in L^1(\mathbb{R}_+,L^{\infty}(\mathbb{R}_+)) \cap L^{\infty}(\mathbb{R}_+,L^{\infty}(\mathbb{R}_+))$
and $h \in L^1(\mathbb{R}_+) \cap L^{\infty}(\mathbb{R}_+)$ are given functions. 
The following lemma gives the exact solutions of the integral 
equations (\ref{eq-Abel-integral}) and (\ref{eq-Abel-integral-local}) 
in the space of bounded functions.

\begin{lemma}
	\label{lem-Abel-integral}
	For every $g \in L^1(\mathbb{R}_+,L^{\infty}(\mathbb{R}_+)) \cap L^{\infty}(\mathbb{R}_+,L^{\infty}(\mathbb{R}_+))$, there exists the unique solution $\gamma \in L^{\infty}(\mathbb{R}_+)$ 
	to the integral equation (\ref{eq-Abel-integral}) in the exact form:
	\begin{equation}
	\label{sol-Abel-integral}
	\gamma(t) = \int_0^t \frac{1}{\sqrt{4 \pi (t-\tau)}} \int_0^{\infty} 
	g(\tau,\eta) \left( \frac{\eta + t - \tau}{2(t - \tau)} \right) e^{-\frac{\eta^2}{4(t-\tau)}} d \eta d\tau, \quad t > 0.
	\end{equation}
For every $h \in L^1(\mathbb{R}_+) \cap L^{\infty}(\mathbb{R}_+)$, there exists the unique solution $\gamma \in L^{\infty}(\mathbb{R}_+)$ 
	to the integral equation (\ref{eq-Abel-integral-local}) in the exact form:
	\begin{equation}
	\label{sol-Abel-integral-local}
	\gamma(t) = \frac{1}{2} h(t) + \frac{1}{2} \int_0^t \frac{h(\tau)  d\tau}{\sqrt{4 \pi (t-\tau)}}, \quad t > 0.
	\end{equation}
\end{lemma}

\begin{proof}
	By using (\ref{laplace-transform}) and (\ref{table-integral}), we solve the integral equation (\ref{eq-Abel-integral}) for the Laplace transform:
$$
\hat{\gamma}(p) = \frac{1}{2} \int_0^{\infty} \hat{g}(p,\eta) \left(1 + \frac{1}{2 \sqrt{p}} \right) e^{-\sqrt{p} \eta} d\eta.
$$
After the inverse Laplace transform, we obtain 
the exact solution (\ref{sol-Abel-integral}) with the use of (\ref{table-integral-derivative}). By using the first integrals 
in (\ref{heat-kernel}) and (\ref{heat-kernel-derivative}), 
we obtain 
$$
\sup_{t \geq 0} |\gamma(t)| \leq \frac{1}{4} \int_0^t  \| g(\tau,\cdot) \|_{L^{\infty}(\mathbb{R}_+)}  d\tau 
+ \frac{1}{2} \int_0^t \frac{\| g(\tau,\cdot) \|_{L^{\infty}(\mathbb{R}_+)}  d\tau}{\sqrt{\pi(t-\tau)}},
$$
where upper bound is bounded if 
$\| g(t,\cdot) \|_{L^{\infty}(\mathbb{R}_+)}$ belongs to $L^1(\mathbb{R}_+) \cap L^{\infty}(\mathbb{R}_+)$.

For the integral equation (\ref{eq-Abel-integral-local}), we use the substitution $\gamma(t) = \frac{1}{2} h(t) + \upsilon(t)$, where 
$\upsilon(t)$ satisfies the integral equation
$$
\mathcal{M}(\upsilon) = \frac{1}{2} \int_0^t \frac{h(\tau)}{\sqrt{4 \pi (t-\tau)}} \int_0^{\infty} 
e^{-\frac{\eta}{2}} e^{-\frac{\eta^2}{4(t-\tau)}} d \eta d \tau, \quad t > 0.
$$
Since $g(\tau,\eta) := \frac{1}{2} h(\tau) e^{-\frac{\eta}{2}}$ belongs to $L^1(\mathbb{R}_+,L^{\infty}(\mathbb{R}_+)) \cap L^{\infty}(\mathbb{R}_+,L^{\infty}(\mathbb{R}_+))$, we can use the exact solution (\ref{sol-Abel-integral}) and obtain 
$$
\upsilon(t) = \frac{1}{2} \int_0^t \frac{h(\tau)}{\sqrt{4 \pi (t-\tau)}} \int_0^{\infty} 
e^{-\frac{\eta}{2}} \left( \frac{\eta + t - \tau}{2(t - \tau)} \right) e^{-\frac{\eta^2}{4(t-\tau)}} d \eta d\tau, \quad t > 0.
$$
Integrating by parts gives 
$$
\upsilon(t) = \frac{1}{2} \int_0^t \frac{h(\tau) d\tau}{\sqrt{4 \pi (t-\tau)}}, \quad t > 0,
$$
which recovers (\ref{sol-Abel-integral-local}) for $\gamma(t) = \frac{1}{2} h(t) + \upsilon(t)$. Again, we have $\gamma \in L^{\infty}(\mathbb{R}_+)$ if $h \in L^1(\mathbb{R}_+) \cap L^{\infty}(\mathbb{R}_+)$.
\end{proof}

\begin{remark}
	Compared to the decomposition method $\gamma = \frac{1}{2} h + \upsilon$ in the proof of Lemma \ref{lem-Abel-integral}, the exact solution (\ref{sol-Abel-integral-local}) can be independently obtained by using the Laplace transform (\ref{laplace-transform}) in the linear equation (\ref{eq-Abel-integral-local}).
\end{remark}

Finally, Young's inequality (\ref{Young}) for convolution integrals in space can be extended to the convolution integrals in time:
\begin{equation}
\label{Young-time}
\| \beta \star \gamma \|_{L^r(\mathbb{R}_+)} \leq \| \beta \|_{L^p(\mathbb{R}_+)} \| \gamma \|_{L^q(\mathbb{R}_+)}, \quad p,q,r \geq 1, \quad 1 + \frac{1}{r} = \frac{1}{p} + \frac{1}{q},
\end{equation}
for every $\beta \in L^p(\mathbb{R}_+)$ and $\gamma \in L^q(\mathbb{R}_+)$, 
where $(\beta \star \gamma)(t) := \int_{0}^{t} \beta(t-\tau) \gamma(\tau) d\tau$ is the convolution integral in time. The following lemma gives useful bounds.

\begin{lemma}
\label{lem-tech}
For every $\gamma \in L^1(\mathbb{R}_+) \cap L^{\infty}(\mathbb{R}_+)$ and every $s \in [0,1)$, there exists a positive constant $C_s$ such that 
\begin{equation}
\label{bound-gamma}
\int_0^t \frac{|\gamma(\tau)|}{(t-\tau)^{s}} d \tau \leq C_s \| \gamma \|_{L^1(\mathbb{R}_+) \cap L^{\infty}(\mathbb{R}_+)}, \quad t > 0. 
\end{equation}
\end{lemma}

\begin{proof}
	For every fixed $T > 0$, it is obvious that 
	$$
	\int_0^t \frac{|\gamma(\tau)|}{(t-\tau)^{s}} d \tau \leq \frac{T^{1-s}}{1-s} \| \gamma \|_{L^{\infty}(\mathbb{R}_+)}, \quad t \in [0,T]. 
	$$
	Then, provided $T > 1$, we get the bounds
\begin{eqnarray*}
\int_0^t \frac{|\gamma(\tau)|}{(t-\tau)^{s}} d \tau 
& = & \int_0^{t-1} \frac{|\gamma(\tau)|}{(t-\tau)^{s}} d \tau + \int_{t-1}^t \frac{|\gamma(\tau)|}{(t-\tau)^{s}} d \tau \\
& \leq & \| \gamma \|_{L^1(\mathbb{R}_+)} + \frac{1}{1-s} \| \gamma \|_{L^{\infty}(\mathbb{R}_+)}, \quad t > T,
\end{eqnarray*}	
and the bound (\ref{bound-gamma}) holds.
\end{proof}

\section{Asymptotic stability under odd perturbations}
\label{section-odd}

Here we study the boundary-value problem (\ref{Burgers-half}) in order to prove Theorem \ref{theorem-1}. The boundary-value problem (\ref{Burgers-half}) is solved by direct methods. First, we decompose
\begin{equation}
\label{decomposition-odd}
w(t,x) = W_0(x) + u(t,x), \quad x > 0,
\end{equation}
where $W_0(x) = 1 - e^{-x}$ is the viscous shock given by (\ref{shock}) with $c = 0$ under the normalization $W_+ = -W_- = 1$. The perturbation $u(t,x)$ satisfies the following boundary-value problem:
\begin{equation}
\label{Burgers-odd}
\left\{ \begin{array}{ll}
u_t = u_x + u_{xx}, \quad & x > 0, \;\; t > 0, \\
u(t,0) = 0, & t > 0, \\
u(t,x) \to 0 \quad \mbox{\rm as} \;\; x \to +\infty, & t > 0, \end{array} \right.
\end{equation}
subject to the initial condition $u(0,x) = w(0,x) - W_0(x) =: u_0(x)$.

In order to prove Theorem \ref{theorem-1}, we first derive a priori energy estimates (Lemma \ref{lemma-1}) and then explore the exact formula  (\ref{u-tilde-sol}) to study the solution in $H^2$ (Lemma \ref{lemma-2}) and in $W^{2,\infty}$ (Lemma \ref{lemma-3}). 

The following lemma implies that the $H^1$-norm of a
smooth solution $u(t,\cdot)$ is decreasing in time $t$. The result is obtained by 
using a priori energy estimates. 

\begin{lemma}
\label{lemma-1} Assume existence of the solution
$u \in C(\mathbb{R}_+,H^2(\mathbb{R}_+))$
to the boundary-value problem (\ref{Burgers-odd}) with the
initial condition $u(0,x) = u_0(x)$. Then, for every $t > 0$:
$$
\| u(t,\cdot) \|_{L^2} \leq \| u_0 \|_{L^2}, \quad
\| u(t,\cdot) \|_{H^1} \leq \| u_0 \|_{H^1}.
$$
\end{lemma}

\begin{proof}
Multiplying $u_t = u_x + u_{xx}$ by $u$ and $u_{xx}$ and integrating by parts yield
\begin{eqnarray}
\label{eq-1}
\frac{d}{dt} \| u(t,\cdot) \|^2_{L^2} & = & -2 \| u_x(t,\cdot) \|^2_{L^2}, \\
\frac{d}{dt} \| u_x(t,\cdot) \|^2_{L^2} & = & \left[ u_x(t,0) \right]^2 -2 \| u_{xx}(t,\cdot) \|^2_{L^2}.
\label{eq-2}
\end{eqnarray}
It follows from (\ref{eq-1}) that $\| u(t,\cdot) \|_{L^2} \leq \| u_0 \|_{L^2}$. By Sobolev embedding, it follows for every $f \in H^1(\mathbb{R}_+)$ that
\begin{equation}
\label{inequality-f}
[f(0)]^2 = -2 \int_0^{\infty} f(x) f'(x) dx \leq \| f \|_{H^1}^2,
\end{equation}
so that we obtain by adding both equations (\ref{eq-1}) and (\ref{eq-2}) 
together and using (\ref{inequality-f}) that 
\begin{eqnarray*}
\frac{d}{dt} \| u(t,\cdot) \|^2_{H^1} = \left[ u_x(t,0) \right]^2 -2 \| u_x(t,\cdot) \|^2_{H^1} \leq -\| u_x (t,\cdot)  \|^2_{H^1},
\end{eqnarray*}
hence $\| u(t,\cdot) \|_{H^1} \leq \| u_0 \|_{H^1}$.
\end{proof}

\begin{remark}
        By using the same method as in the proof of Lemma \ref{lemma-1}, one can derive 
	\begin{equation}
	\label{eq-3}
	\frac{d}{dt} \| u_{xx}(t,\cdot) \|^2_{L^2} = -\left[ u_{xx}(t,0) \right]^2 - 2 u_{xx}(t,0) u_{xxx}(t,0) - 2 \| u_{xxx}(t,\cdot) \|^2_{L^2}.
	\end{equation}
	By using also $u_x(t,0) + u_{xx}(t,0) = 0$ and $u_{tx}(t,0) = u_{xx}(t,0) + u_{xxx}(t,0)$ for smooth solutions, this balance equation can be rewritten to the form:
		\begin{equation}
	\label{eq-3a}
	\frac{d}{dt} \left( \| u_{xx}(t,\cdot) \|^2_{L^2} - [u_x(t,0)]^2 \right) = \left[ u_{x}(t,0) \right]^2 - 2 \| u_{xxx}(t,\cdot) \|^2_{L^2}.
	\end{equation}
	With an inequality similar to (\ref{inequality-f}), we derive 
	$$
	\| u(t,\cdot) \|_{H^2}^2 - [u_x(t,0)]^2 \leq \| u_0 \|_{H^2}^2 - [u_0'(0)]^2.
	$$
	However, due to the inequality (\ref{inequality-f}) this a priori energy estimate does not imply monotonicity of the $H^2$-norm of the smooth solutions of the boundary-value problem (\ref{Burgers-odd}).
\end{remark}

Lemma \ref{lemma-1} implies uniqueness and continuous dependence of solutions to the boundary-value problem (\ref{Burgers-odd}) with initial condition $u(0,x) = u_0(x)$. It remains to show existence of a solution $u \in C(\mathbb{R}_+,H^2(\mathbb{R}_+))$ for any given $u_0 \in H^2(\mathbb{R}_+)$.
The following lemma explores an explicit formula for solutions $u \in C(\mathbb{R}_+,H^2(\mathbb{R}_+))$ to the boundary-value problem (\ref{Burgers-odd}) for any given initial condition $u_0 \in H^2(\mathbb{R}_+)$.

\begin{lemma}
\label{lemma-2}
For any given $u_0 \in H^2(\mathbb{R}_+)$, there exists a solution
$u(t,x)$ to the boundary-value problem (\ref{Burgers-odd}) with the
initial condition $u(0,x) = u_0(x)$ given explicitly by
\begin{equation}
\label{sol-explicit}
u(t,x) = \frac{1}{\sqrt{4 \pi t}} \int_0^{\infty} u_0(y) \left[ e^{-\frac{(x-y+t)^2}{4t}} - e^{-x} e^{-\frac{(x+y-t)^2}{4t}} \right] dy.
\end{equation}
Moreover, $u \in C(\mathbb{R}_+,H^2(\mathbb{R}_+))$.
\end{lemma}

\begin{proof}
By using the transformation
\begin{equation}
\label{u-tilde-u}
u(t,x) = e^{-\frac{x}{2}-\frac{t}{4}} v(t,x),
\end{equation}
we can write the boundary-value problem (\ref{Burgers-odd}) in the form (\ref{Burgers-odd-tilde}) 
with the initial condition $v_0(x) = e^{\frac{x}{2}} u_0(x)$, where $u_0(x) = u(0,x)$. By substituting the transformation (\ref{u-tilde-u}) to the exact solution (\ref{u-tilde-sol})
and completing squares for the heat kernel $G(t,x) = \frac{1}{\sqrt{4\pi t}} e^{-\frac{x^2}{4t}}$, we obtain the exact representation
(\ref{sol-explicit}). 

Next we show that  $u \in C(\mathbb{R}_+,H^2(\mathbb{R}_+))$ if $u_0 \in H^2(\mathbb{R}_+)$. The convolution integrals in (\ref{sol-explicit}) are analyzed by means of the
generalized Young's inequality (\ref{Young}) with $p = r = 2$ and $q = 1$:
\begin{eqnarray*}
\| u_0 \chi_{\mathbb{R}_+} \ast 
G(t,\cdot + t) \|_{L^2(\mathbb{R}_+)}  \leq  \| u_0 \|_{L^2(\mathbb{R}_+)} \| G(t,\cdot + t) \|_{L^1(\mathbb{R})} 
\leq \| u_0 \|_{L^2(\mathbb{R}_+)}
\end{eqnarray*}
and
\begin{eqnarray*}
\| u_0 \chi_{\mathbb{R}_+} \ast 
G(t,-\cdot + t) \|_{L^2(\mathbb{R}_+)} \leq \| u_0 \|_{L^2(\mathbb{R}_+)} \| G(t,-\cdot + t) \|_{L^1(\mathbb{R})} \leq \| u_0 \|_{L^2(\mathbb{R}_+)}
\end{eqnarray*}
At the same time, $e^{-x} \leq 1$ for $x \geq 0$, so that 
\begin{equation}
\label{bound-derivative-zero}
\| u(t,\cdot) \|_{L^2} \leq 2 \| u_0 \|_{L^2},
\end{equation}
where the $L^2$ norms are understood as $L^2(\mathbb{R}_+)$.
In order to obtain similar estimates for $u_x$ and $u_{xx}$,
we differentiate (\ref{sol-explicit}) in $x$, use integration by parts, and obtain
\begin{eqnarray}
\nonumber
u_x(t,x) & = & \frac{1}{\sqrt{4 \pi t}} \int_0^{\infty} u_0'(y) \left[ e^{-\frac{(x-y+t)^2}{4t}} + e^{-x} e^{-\frac{(x+y-t)^2}{4t}} \right] dy\\
\label{sol-derivative}
&& 
+ \frac{1}{\sqrt{4 \pi t}} e^{-x} \int_0^{\infty} u_0(y) e^{-\frac{(x+y-t)^2}{4t}} dy
\end{eqnarray}
and
\begin{eqnarray}
\nonumber
u_{xx}(t,x) & = & \frac{1}{\sqrt{4 \pi t}} \int_0^{\infty} u_0''(y) \left[ e^{-\frac{(x-y+t)^2}{4t}} -  e^{-x} e^{-\frac{(x+y-t)^2}{4t}} \right] dy \\
\label{sol-second-derivative}
& \phantom{t} & - \frac{1}{\sqrt{\pi t}} e^{-x} \int_0^{\infty} u_0'(y) e^{-\frac{(x+y-t)^2}{4t}} dy -
\frac{1}{\sqrt{4 \pi t}} e^{-x} \int_0^{\infty} u_0(y) e^{-\frac{(x+y-t)^2}{4t}} dy, \qquad \qquad 
\end{eqnarray}
where the boundary condition $u(t,0) = 0$ has been used.
By the same estimates used in (\ref{bound-derivative-zero}), we obtain:
\begin{eqnarray}
\label{bound-derivative-first}
\| u_x(t,\cdot) \|_{L^2} & \leq & 2 \| u_0' \|_{L^2} + \| u_0 \|_{L^2}, \\
\label{bound-derivative-second}
\| u_{xx}(t,\cdot) \|_{L^2} & \leq & 2 \| u_0'' \|_{L^2} + 2 \| u_0' \|_{L^2} + \| u_0 \|_{L^2}.
\end{eqnarray}
This shows that $u(t,\cdot) \in H^2(\mathbb{R}_+)$ continuously in $t \in \mathbb{R}_+$. It follows from (\ref{sol-derivative}) and (\ref{sol-second-derivative}) as $x \to 0^+$
that 
\begin{equation}
\label{bc-odd-satisfied}
u_{x}(t,0^+) + u_{xx}(t,0^+) = 0, \quad t > 0.
\end{equation}
The decay condition $u(t,x) \to 0$ as $x \to \infty$ is satisfied by
the continuous embedding of $H^2(\mathbb{R}_+)$ into $C^1(\mathbb{R}_+) \cap W^{1,\infty}(\mathbb{R}_+)$ with functions and their first derivatives decaying to zero at infinity.
\end{proof}

The following lemma establishes the decay of $\| u(t,\cdot) \|_{W^{2,\infty}}$ to zero as $t \to +\infty$.

\begin{lemma}
\label{lemma-3}
Let $u \in C(\mathbb{R}_+,H^2(\mathbb{R}_+))$ be the solution to the boundary-value problem (\ref{Burgers-odd}) given by Lemma \ref{lemma-2}.
Then, we have
\begin{equation}
\label{asym-stability-decay}
\| u(t,\cdot) \|_{W^{2,\infty}} \to 0 \quad \mbox{\rm as} \quad t \to \infty.
\end{equation}
\end{lemma}

\begin{proof}
For $t \geq 1$, we can estimate the convolution integrals in (\ref{sol-explicit}) by means of the
generalized Young's inequality (\ref{Young}) with $p = q = 2$ and $r = \infty$:
\begin{eqnarray*}
\| u_0 \chi_{\mathbb{R}_+} \ast 
G(t,\cdot + t) \|_{L^{\infty}(\mathbb{R}_+)} 
\leq \| u_0 \|_{L^2(\mathbb{R}_+)} \| G(t,\cdot + t) \|_{L^2(\mathbb{R})}
\leq \frac{1}{(8\pi t)^{1/4}} \| u_0 \|_{L^2(\mathbb{R}_+)}
\end{eqnarray*}
and
\begin{eqnarray*}
\| u_0 \chi_{\mathbb{R}_+} \ast 
G(t, - \cdot + t) \|_{L^{\infty}(\mathbb{R}_+)} \leq 
\| u_0 \|_{L^2(\mathbb{R}_+)} \| G(t,-\cdot + t) \|_{L^2(\mathbb{R})} 
\leq \frac{1}{(8\pi t)^{1/4}} \| u_0 \|_{L^2(\mathbb{R}_+)}.
\end{eqnarray*}
Using these estimates in (\ref{sol-explicit}), (\ref{sol-derivative}), and (\ref{sol-second-derivative}), we obtain
\begin{eqnarray*}
\| u(t,\cdot) \|_{L^{\infty}} & \leq & \frac{2}{(8\pi t)^{1/4}} \| u_0 \|_{L^2}, \\
\| u_x(t,\cdot) \|_{L^{\infty}} & \leq & \frac{1}{(8\pi t)^{1/4}} (2 \| u_0' \|_{L^2} + \| u_0 \|_{L^2}), \\
\| u_{xx}(t,\cdot) \|_{L^{\infty}} & \leq & \frac{1}{(8\pi t)^{1/4}} (2 \| u_0'' \|_{L^2} + 2 \| u_0' \|_{L^2} + \| u_0 \|_{L^2}),
\end{eqnarray*}
which prove the decay (\ref{asym-stability-decay}).
\end{proof}

\begin{remark}
	The asymptotic decay of the solution $u \in C(\mathbb{R}_+,H^2(\mathbb{R}_+))$ in the $H^2$ norm does not follow from the convolution estimates (\ref{Young}) unless $u_0^{\pm} \in W^{2,p}(\mathbb{R}_+)$ for some $p < 2$.
	\label{remark-decay-H2}
\end{remark}

\begin{proof1}{\em of Theorem \ref{theorem-1}.}
By Lemma \ref{lemma-2} and the bounds (\ref{bound-derivative-zero}), 
(\ref{bound-derivative-first}), and (\ref{bound-derivative-second}),
if $u_0 \in H^2(\mathbb{R}_+)$ satisfies $\| u_0 \|_{H^2} < \delta$ as
in (\ref{initial-1}), then
$$
\| u(t,\cdot) \|_{H^2} \leq C \| u_0 \|_{H^2} < C \delta
$$
for a fixed $\delta$-independent positive constant $C$. Hence, for every $\epsilon > 0$,
there is $\delta := \epsilon/C$ such that the odd perturbation $u(t,\cdot)$ to the viscous shock
$W_0$ in the decomposition (\ref{decomposition-odd}) is bounded in $H^2(\mathbb{R})$ norm for every $t > 0$ according to the bound (\ref{final-1}). 
The decay (\ref{scattering-1}) 
follows from the decay (\ref{asym-stability-decay}) in Lemma \ref{lemma-3}.

The constraint (\ref{interface-half}) is satisfied because 
both $W_0$ and $u$ in the decomposition (\ref{decomposition-odd}) satisfy this constraint. 
Under the constraint (\ref{interface-half}), the solutions 
$w(t,x) : \mathbb{R}_+ \times \mathbb{R}_+ \mapsto \mathbb{R}$ to the boundary-value problem (\ref{Burgers-half}) are extended to the 
odd function $w_{\rm ext}(t,x) : \mathbb{R}_+ \times \mathbb{R}\mapsto \mathbb{R}$ satisfying the interface condition (\ref{interface-half-extended}). 

It remains to verify that $w(t,x) = W_0(x) + u(t,x) > 0$ for every $x > 0$. The positivity
condition (\ref{w-positivity}) is necessary for reduction of the modular Burgers equation (\ref{Burgers})
with the odd functions to the boundary-value problem (\ref{Burgers-half}). By Sobolev embedding of
$H^2(\mathbb{R}_+)$ into $C^1(\mathbb{R}_+) \cap W^{1,\infty}(\mathbb{R}_+)$, we obtain the uniform bound:
\begin{eqnarray*}
\| u(t,\cdot) \|_{L^{\infty}} +  \| u_x(t,\cdot) \|_{L^{\infty}} < \epsilon, \quad t > 0,
\end{eqnarray*}
where $\epsilon$ is small. The symmetry point $x = 0$ is a simple root of $w(t,\cdot)$ for every $t > 0$
because $W_0(0) = 0$, $W_0'(0) = 1$, $u(t,0) = 0$, and $|u_x(t,0)| < \epsilon$ is small.
Therefore, there exists an $\epsilon$-independent $x_0 > 0$ such that $w(t,x) > 0$ for every $t > 0$ and $x \in (0,x_0)$.
Now, $W_0(x) \geq W_0(x_0) > 0$ for every $x \geq x_0$ and since
$|u(t,x)| < \epsilon$ for every $t > 0$ and $x > 0$, then $w(t,x) > 0$ for every $t > 0$ and $x \geq x_0$ if $\epsilon$ is sufficiently small.
Combining these two estimates together yields 
$w(t,x) > 0$ for every $t > 0$ and $x > 0$. 
\end{proof1}

\section{Asymptotic stability under general perturbations}
\label{section-general}

Here we study the boundary-value problem (\ref{Burgers-interface}) in order to prove Theorem \ref{theorem-2}. The boundary-value problem (\ref{Burgers-interface}) can be reformulated by using the decomposition
\begin{equation}
\label{decomposition-interface}
w(t,x) = W_0(x-\xi(t)) + u(t,x - \xi(t)), \quad x \in \mathbb{R},
\end{equation}
where $W_0$ is the viscous shock (\ref{shock}) with $c = 0$ under the normalization
$W_+ = -W_- = 1$, $\xi(t)$ is the location of a single interface,
and $u(t,y)$ with $y := x - \xi(t)$ is a perturbation satisfying
\begin{equation}
\label{Burgers-general}
\left\{ \begin{array}{l}
u_t = (\xi'(t) \pm 1) u_y + u_{yy} + \xi'(t) W_0'(y), \quad \pm y > 0, \\
u(t,0) = 0, \\
u(t,y) \to 0 \quad \mbox{\rm as} \;\; y \to \pm \infty, \end{array} \right.
\end{equation}
subject to the initial condition $u(0,x) = w(0,x) - W_0(x) =: u_0(x)$. We assume
without loss of generality that $\xi(0) = 0$. The interface dynamics is defined
by the following lemma.

\begin{lemma}
\label{lem-interface}
Let $u(t,\cdot) \in C^1(\mathbb{R}) \cap C^2(\mathbb{R} \backslash \{0\})$
be a solution of the boundary-value problem (\ref{Burgers-general}) for $t \in \mathbb{R}_+$.
Then, $\xi'(t)$, $t \in \mathbb{R}_+$ can be expressed in two equivalent ways by
\begin{equation}
\label{dynamics-general}
\xi'(t) = -\frac{u_y(t,0^+) + u_{yy}(t,0^+)}{1 + u_y(t,0^+)} = \frac{u_y(t,0^-) - u_{yy}(t,0^-)}{1 + u_y(t,0^-)}, \quad t \in \mathbb{R}_+.
\end{equation}
\end{lemma}

\begin{proof}
It follows from (\ref{interface-interface}) and (\ref{decomposition-interface})
that piecewise $C^2$ solutions of the boundary-value problem
(\ref{Burgers-general}) satisfy the interface condition
\begin{equation}
\label{interface-general}
[u_{yy}]^+_-(t,0) = - 2 u_y(t,0).
\end{equation}
On the other hand, it follows from (\ref{interface-continuity}) and (\ref{decomposition-interface})
that $u_t(t,0) = 0$. Taking the limits $y \to 0^{\pm}$ results in the dynamical equations (\ref{dynamics-general}) since $W_0'(0) = 1$. 
The two equalities in (\ref{dynamics-general}) are consistent under the interface condition (\ref{interface-general}) since $u_y(t,0^+) = u_y(t,0^-)$.
\end{proof}

\begin{remark}
The system of equations (\ref{Burgers-general}), (\ref{dynamics-general}), and (\ref{interface-general}) is derived under the condition
\begin{equation}
\label{w-positivity-broken}
\pm \left[ W_0(y) + u(t,y) \right] > 0, \quad \pm y > 0
\end{equation}
which follows from (\ref{w-positivity-general}) and (\ref{decomposition-interface}). Since $W_0(0) = 0$, 
$W_0'(0) = 1$, and $u(t,0) = 0$, 
the positivity conditions (\ref{w-positivity-broken}) imply that $1 + u_y(t,0) > 0$,
hence the interface dynamics is well defined by 
the evolution equation (\ref{dynamics-general}) under 
the positivity conditions (\ref{w-positivity-broken}).
\end{remark}

Let us define 
\begin{equation}
\label{variables-plus-minus}
u^+(t,y) := u(t,y), \qquad u^-(t,y) := u(t,-y), \qquad y > 0.
\end{equation}
We also define $\gamma(t) := \xi'(t)$ and use $W_0'(y) = e^{-|y|}$.
The boundary-value problem (\ref{Burgers-general})
can be rewritten in the equivalent form
\begin{equation}
\label{u-eqs}
\left\{ \begin{array}{l}
u^{\pm}_t = (1 \pm \gamma) u^{\pm}_y + u^{\pm}_{yy} + \gamma e^{-y}, \quad y > 0, \\
u^{\pm}(t,0) = 0, \\
u^{\pm}(t,y) \to 0 \quad \mbox{\rm as} \;\; y \to \infty, \end{array} \right.
\end{equation}
subject to the continuity condition
\begin{equation}
\label{u-continuity}
u^+_y(t,0^+) = - u^-_y(t,0^+),
\end{equation}
the interface condition
\begin{equation}
\label{u-redundant-cond}
u^+_{yy}(t,0^+) - u^-_{yy}(t,0^+) = -2 u^+_y(t,0^+),
\end{equation}
and the dynamical condition
\begin{equation}
\label{u-xi-dot}
\gamma(t) = -\frac{u^+_y(t,0^+) + u_{yy}^+(t,0^+)}{1 + u_y^+(t,0^+)} = -\frac{u^-_y(t,0^+) + u^-_{yy}(t,0^+)}{1 - u^-_y(t,0^+)}.
\end{equation}

The proof of Theorem \ref{theorem-2} is divided into two steps. 

In the first step, for a given $\gamma \in L^1(\mathbb{R}_+) \cap L^{\infty}(\mathbb{R}_+)$, 
we show that the boundary-value problems  (\ref{u-eqs}) equipped with the initial 
conditions $u^{\pm}(0,x) = u_0^{\pm}(x)$ can be uniquely solved 
provided the norms of $\gamma \in L^1(\mathbb{R}_+) \cap L^{\infty}(\mathbb{R}_+)$ and $u_0^{\pm} \in H^2(\mathbb{R}_+) \cap W^{2,\infty}(\mathbb{R}_+)$ are small. 
If $\gamma \in C(\mathbb{R}_+)$, the unique global solutions $u^{\pm} \in C(\mathbb{R}_+,H^2(\mathbb{R}_+) \cap W^{2,\infty}(\mathbb{R}_+))$ satisfy the dynamical conditions 
(\ref{u-xi-dot}) for any $t > 0$.

The two solutions for $u^+$ and $u^-$ are uncoupled if 
$\gamma$ is given. However, if the solutions $u^+$ and $u^-$ are required to satisfy the continuity condition
(\ref{u-continuity}), then this constraint yields an integral 
equation on $\gamma \in  L^1(\mathbb{R}_+) \cap L^{\infty}(\mathbb{R}_+)$. In the second step, we prove that the integral equation for $\gamma \in  L^1(\mathbb{R}_+) \cap L^{\infty}(\mathbb{R}_+)$ can be uniquely solved provided $u_0^{\pm} \in H^2(\mathbb{R}_+) \cap W^{2,\infty}(\mathbb{R}_+)$ are small and satisfy an additional requirement of the exponential decay in space.

Finally, 
the two conditions (\ref{u-continuity}) and (\ref{u-xi-dot})  imply the interface condition (\ref{u-redundant-cond}), which is thus redundant in the boundary-value problem.

The following lemma gives a priori energy estimates for the boundary-value problems (\ref{u-eqs}) completed with the continuity condition (\ref{u-continuity}). These energy estimates imply monotonicity of the $H^1$-norm of a smooth solution in time $t$.

\begin{lemma}
\label{lemma-gen-1}
Assume existence of the solutions $u^{\pm} \in C(\mathbb{R}_+,H^2(\mathbb{R}_+))$ to the boundary-value problem (\ref{u-eqs}) completed with the continuity condition (\ref{u-continuity}) 
for the initial conditions $u^{\pm}(0,x) = u_0^{\pm}(x)$ and for some $\gamma \in C(\mathbb{R}_+)$. Then, for every $t > 0$:
\begin{equation}
\label{v-plus-minus-H-1}
\| u^+(t,\cdot) \|^2_{H^1} + \| u^-(t,\cdot) \|^2_{H^1} \leq \| u_0^+ \|^2_{H^1} + \| u_0^- \|^2_{H^1}.
\end{equation}
\end{lemma}

\begin{proof}
Multiplying $u_t^{\pm} = (1 \pm \gamma) u_y^{\pm} + u_{yy}^{\pm} + \gamma e^{-y}$ by $u^{\pm}$ and $u_{yy}^{\pm}$ and integrating by parts yield
\begin{eqnarray}
\label{first-eq}
\frac{d}{dt} \| u^{\pm}(t,\cdot) \|^2_{L^2} & = & -2 \| u^{\pm}_y(t,\cdot) \|^2_{L^2} + 2 \gamma \int_0^{\infty} u^{\pm}(t,y) e^{-y} dy, \\
\nonumber
\frac{d}{dt} \| u^{\pm}_y(t,\cdot) \|^2_{L^2} & = & (1 \pm \gamma) \left[ u_y^{\pm}(t,0^+) \right]^2 - 2 \| u^{\pm}_{yy}(t,\cdot) \|^2_{L^2} 
+ 2 \gamma u_y^{\pm}(t,0^+) - 2 \gamma \int_0^{\infty} u^{\pm}(t,y) e^{-y} dy.
\end{eqnarray}
Adding all equations and using the continuity condition (\ref{u-continuity}) yield
\begin{eqnarray*}
\frac{d}{dt} \| u^+(t,\cdot) \|^2_{H^1} +  \frac{d}{dt} \| u^-(t,\cdot) \|^2_{H^1} =
\left[ u_y^+(t,0^+) \right]^2 + \left[ u_y^-(t,0^+) \right]^2 - 2 \| u^+_y(t,\cdot) \|^2_{H^1} - 2 \| u^-_y(t,\cdot) \|^2_{H^1}.
\end{eqnarray*}
By using the same inequality (\ref{inequality-f}), we close the estimate and obtain
\begin{eqnarray*}
\frac{d}{dt} \left[ \| u^+(t,\cdot) \|^2_{H^1} +  \| u^-(t,\cdot) \|^2_{H^1} \right] \leq
- \| u^+_y(t,\cdot) \|^2_{H^1} - 2 \| u^-_y(t,\cdot) \|^2_{H^1} \leq 0,
\end{eqnarray*}
from which the inequality (\ref{v-plus-minus-H-1}) follows.
\end{proof}

\begin{remark}
Compared to Lemma \ref{lemma-1}, we are not able to conclude on monotonicity of the $L^2$-norm of the solution. Integrating by parts and using Cauchy--Schwarz inequality in (\ref{first-eq}), we get 
$$
\frac{d}{dt} \| u^{\pm}(t,\cdot) \|_{L^2} \leq 
|\gamma| \| e^{-y} \|_{L^2_y(\mathbb{R}_+)},
$$
which yields the Stritcharz-type estimate
\begin{eqnarray*}
\sup_{t \in \mathbb{R}_+} \| u^{\pm}(t,\cdot) \|_{L^2} \leq
\| u_0^{\pm} \|_{L^2} + \| \gamma \|_{L^1} \| e^{-y} \|_{L^2_y(\mathbb{R}_+)},
\end{eqnarray*}
where we write $\| e^{-y} \|_{L^2_y(\mathbb{R}_+)}$ instead of 
$\| e^{-\cdot} \|_{L^2}$ for better clarity.
\end{remark}

We shall now consider the existence of solutions for the boundary-value problems (\ref{u-eqs}) with $\gamma$ expressed by (\ref{u-xi-dot}). Due to the latter condition, we need to require the second derivative to be bounded and piecewise continuous
in a one-sided neighborhood of $y = 0$. This is achieved by 
using a sharper condition on the initial data $u_0^{\pm} \in H^2(\mathbb{R}_+) \cap W^{2,\infty}(\mathbb{R}_+)$ compared to the requirement 
$u_0 \in H^2(\mathbb{R}_+)$ imposed in Lemma \ref{lemma-2}.

The following lemma provides a convenient
reformulation of the boundary-value problems (\ref{u-eqs}) as systems of integral equations.
In these systems, $u^{\pm}$ and $\gamma$ are not required to 
satisfy the continuity condition (\ref{u-continuity}), 
the interface condition (\ref{u-redundant-cond}), and the dynamical conditions (\ref{u-xi-dot}).

\begin{lemma}
\label{lemma-gen-2}
There exist solutions $u^{\pm} \in C(\mathbb{R}_+,H^2(\mathbb{R}_+) \cap W^{2,\infty}(\mathbb{R}_+))$ to the boundary-value problems (\ref{u-eqs}) with the initial conditions $u^{\pm}(0,x) = u_0^{\pm}(x)$ and the given function $\gamma \in C(\mathbb{R}_+)$ if there exist solutions $u^{\pm} \in C(\mathbb{R}_+,H^2(\mathbb{R}_+) \cap W^{2,\infty}(\mathbb{R}_+))$ to the following integral equations for $(t,y) \in \mathbb{R}_+ \times \mathbb{R}_+$:
\begin{eqnarray*}
u^{\pm}(t,y) & = & \frac{1}{\sqrt{4 \pi t}} \int_0^{\infty} u_0^{\pm}(\eta) \left[
e^{-\frac{(y-\eta + t)^2}{4t}} - e^{-y} e^{-\frac{(y+\eta - t)^2}{4t}} \right] d \eta \\
& \phantom{t} & + \int_0^t \frac{\gamma(\tau) d \tau}{\sqrt{4 \pi (t-\tau)}}
\int_0^{\infty} e^{-\eta} \left[
e^{-\frac{(y-\eta + t - \tau)^2}{4(t - \tau)}} - e^{-y} e^{-\frac{(y+\eta - t + \tau)^2}{4 (t-\tau)}} \right] d\eta \\
& \phantom{t} & \pm  \int_0^t \frac{\gamma(\tau) d \tau}{\sqrt{4 \pi (t-\tau)}}
\int_0^{\infty} u_{\eta}^{\pm}(\tau,\eta) \left[
e^{-\frac{(y-\eta + t - \tau)^2}{4(t - \tau)}} - e^{-y} e^{-\frac{(y+\eta - t + \tau)^2}{4 (t-\tau)}} \right] d\eta.
\end{eqnarray*}
\end{lemma}

\begin{proof}
Similar to the transformation formula (\ref{u-tilde-u}) in the proof of Lemma \ref{lemma-2},
the system of equations (\ref{u-eqs}) can be simplified by using the transformation formulas:
\begin{equation}
\label{v-tilde-v}
u^{\pm}(t,y) = e^{-\frac{y}{2}-\frac{t}{4}} v^{\pm}(t,y), \quad \gamma(t) = e^{-\frac{t}{4}} \tilde{\gamma}(t).
\end{equation}
The boundary-value problems (\ref{u-eqs}) can be rewritten in the form (\ref{Burgers-odd-tilde-inhomog}) with $v = v^{\pm}$, 
$$
f(t,y) = \tilde{\gamma} e^{-\frac{y}{2}} \pm \tilde{\gamma} e^{-\frac{t}{4}} (v^{\pm}_y - \frac{1}{2} v^{\pm}),
$$
and $v_0(y) :=  u^{\pm}_0(y) e^{\frac{y}{2}}$, where $u_0^{\pm}(y) := u^{\pm}(0,y)$ are the initial conditions. By using the exact solution (\ref{u-tilde-sol-inhomog}), we obtain the integral equations for 
$v^{\pm}$:
\begin{eqnarray*}
v^{\pm}(t,y) & = & \frac{1}{\sqrt{4 \pi t}} \int_0^{\infty} v_0^{\pm}(\eta)
\left[ e^{-\frac{(y-\eta)^2}{4t}} - e^{-\frac{(y+\eta)^2}{4t}} \right] d \eta \\
& \phantom{t} & +  \int_0^t \frac{\tilde{\gamma}(\tau) d \tau}{\sqrt{4 \pi (t-\tau)}} \int_0^{\infty} e^{-\frac{\eta}{2}} \left[
e^{-\frac{(y-\eta)^2}{4(t - \tau)}} - e^{-\frac{(y+\eta)^2}{4 (t-\tau)}} \right] d\eta \\
& \phantom{t} & \pm \int_0^t \frac{\tilde{\gamma}(\tau) e^{-\frac{\tau}{4}} d \tau}{\sqrt{4 \pi (t-\tau)}} \int_0^{\infty}
\left( v^{\pm}_y(\tau,\eta) - \frac{1}{2} v^{\pm}(\tau,\eta)\right)  \left[
e^{-\frac{(y-\eta)^2}{4(t - \tau)}} - e^{-\frac{(y+\eta)^2}{4 (t-\tau)}} \right] d\eta.
\end{eqnarray*}
Substituting the transformation (\ref{v-tilde-v}) yields the integral equations for $u^{\pm}(t,y)$.
\end{proof}

Given solutions $u^{\pm} \in C(\mathbb{R}_+,H^2(\mathbb{R}_+) \cap W^{2,\infty}(\mathbb{R}_+))$ of the integral equations in 
Lemma \ref{lemma-gen-2}, we require them to satisfy the continuity condition (\ref{u-continuity}). This sets up the existence problem for $\gamma \in L^1(\mathbb{R}_+) \cap L^{\infty}(\mathbb{R}_+)$. By computing partial derivatives of $u^{\pm}(t,y)$ in $y$, taking the limit $y \to 0^+$,  substituting $u_y^{\pm}(t,0^+)$ into 
(\ref{u-continuity}), and integrating by parts, we obtain  the following integral equation:
\begin{eqnarray}
\nonumber
&& \int_0^t \frac{\gamma(\tau)}{\sqrt{\pi (t-\tau)}}
	e^{-\frac{t - \tau}{4}}  d \tau - 
	\int_0^t \frac{\gamma(\tau)}{\sqrt{4 \pi (t-\tau)}}
	\int_0^{\infty} e^{-\frac{(\eta + t - \tau)^2}{4(t - \tau)}} d\eta  d \tau \\
\label{eq-gamma}
&&  + \frac{1}{\sqrt{4 \pi t}} \int_0^{\infty} \left[ u_0^{+ \prime}(\eta) + u_0^{- \prime}(\eta) + \frac{1}{2} u_0^+(\eta) + \frac{1}{2} u_0^-(\eta) \right]
	e^{-\frac{(\eta - t)^2}{4t}} d \eta \\
\nonumber
&& + \int_0^t \frac{\gamma(\tau)}{\sqrt{4 \pi (t-\tau)}}
	\int_0^{\infty} \left[  u_y^+ - u_y^- + \frac{1}{2} u^+ - \frac{1}{2} u^- \right](\tau,\eta) \left(
	\frac{\eta - t + \tau}{2(t-\tau)} \right)
	e^{-\frac{(\eta - t + \tau)^2}{4 (t-\tau)}} d\eta  d \tau = 0.
\end{eqnarray}
The following lemma rewrites the integral equation (\ref{eq-gamma}) 
in the equivalent form.

\begin{lemma}
	\label{lemma-gen-4}
	There exists a solution $\gamma \in L^1(\mathbb{R}_+) \cap L^{\infty}(\mathbb{R}_+)$ to the integral equation (\ref{eq-gamma}) 
	if there exists 
	a solution $\gamma \in L^1(\mathbb{R}_+) \cap L^{\infty}(\mathbb{R}_+)$ to the following integral equation 
	\begin{eqnarray*}
\gamma(t) & = & -\frac{1}{\sqrt{4\pi t}} \int_0^{\infty} 
		\left[ u_0^{+ \prime}(\eta) + u_0^{- \prime}(\eta) + \frac{1}{2} u_0^+(\eta) + \frac{1}{2} u_0^-(\eta) \right] 
		\left( \frac{\eta + t}{2 t} \right) e^{-\frac{(\eta-t)^2}{4t}} d \eta	\\
		&& 
		- \frac{1}{2} \gamma(t) \left[ u_y^+(t,0^+) - u_y^-(t,0^+) \right] 
		- \frac{1}{2}
		\int_0^t \frac{\gamma(\tau) e^{-\frac{t -\tau}{4}}}{\sqrt{4 \pi (t-\tau)}}  \left[ u_{y}^+(\tau,0^+) - u_{y}^-(\tau,0^+) \right] d \tau \\ 
		&& - \int_0^t \frac{\gamma(\tau)}{\sqrt{4 \pi (t-\tau)}} \int_0^{\infty} \left[ u_{yy}^+ - u_{yy}^- + \frac{1}{2} u_y^+ - \frac{1}{2} u_y^- \right](\tau,\eta) \left( \frac{\eta + t - \tau}{2 (t-\tau)} \right) e^{-\frac{(\eta - t + \tau)^2}{4(t-\tau)}} d\eta d \tau.	
	\end{eqnarray*}
\end{lemma}

\begin{proof}
	By using the same transformation (\ref{v-tilde-v}), we rewrite the integral equation (\ref{eq-gamma}) in the equivalent form:
	\begin{eqnarray}
	\nonumber
&&	\mathcal{M}(\tilde{\gamma}) + \frac{1}{\sqrt{4 \pi t}} \int_0^{\infty} \left[ v_0^{+ \prime}(\eta) +  v_0^{- \prime}(\eta) \right]
		e^{-\frac{\eta^2}{4t}} d \eta \\
&& + \int_0^t \frac{\tilde{\gamma}(\tau) e^{-\frac{\tau}{4}}}{\sqrt{4 \pi (t-\tau)}}
		\int_0^{\infty} \left[ v_y^+ - v_y^- \right](\tau,\eta) \left(
		\frac{\eta - t + \tau}{2(t-\tau)} \right)
		e^{-\frac{\eta^2}{4 (t-\tau)}} d\eta d\tau = 0,
		\label{M-integral-eq}
	\end{eqnarray}	
where the linear operator $\mathcal{M}$ is given by (\ref{operator-M}) 
and $v_0^{\pm}(y) = u_0^{\pm}(y) e^{\frac{y}{2}}$. By
using Lemma \ref{lem-Abel} with 
$$
f(\eta) = v_0^{+ \prime}(\eta) +  v_0^{- \prime}(\eta), 
$$
the linear operator $\mathcal{M}$ can be inverted on the second term of the integral equation (\ref{M-integral-eq}). In order to invert the linear operator $\mathcal{M}$ on the third term of the integral equation (\ref{M-integral-eq}),  we integrate it by parts and obtain
	\begin{eqnarray*}
		&& \int_0^t \frac{\tilde{\gamma}(\tau) e^{-\frac{\tau}{4}}}{\sqrt{4 \pi (t-\tau)}}
		\int_0^{\infty} \left[ v_y^+ - v_y^- \right](\tau,\eta) \left(
		\frac{\eta - t + \tau}{2(t-\tau)} \right)
		e^{-\frac{\eta^2}{4 (t-\tau)}} d\eta d\tau \\
		&& = \int_0^t \frac{\tilde{\gamma}(\tau) e^{-\frac{\tau}{4}}}{\sqrt{4 \pi (t-\tau)}} \left[ v_y^+(\tau,0^+) - v_y^-(\tau,0^+) \right] d\tau \\
		&& + 
		\int_0^t \frac{\tilde{\gamma}(\tau) e^{-\frac{\tau}{4}}}{\sqrt{4 \pi (t-\tau)}}
		\int_0^{\infty} \left[ v_{yy}^+ - v_{yy}^- - \frac{1}{2} v_y^+ + \frac{1}{2} v_y^- \right](\tau,\eta) e^{-\frac{\eta^2}{4 (t-\tau)}} d\eta d\tau.
	\end{eqnarray*}
We are now in position to use Lemma \ref{lem-Abel-integral} with 
$$
g(\tau,\eta) = \tilde{\gamma}(\tau) e^{-\frac{\tau}{4}} \left[ v_{yy}^+(\tau,\eta) - v_{yy}^-(\tau,\eta) - \frac{1}{2} v_y^+(\tau,\eta) + \frac{1}{2} v_y^-(\tau,\eta) \right]
$$
and
$$
h(\tau) = \tilde{\gamma}(\tau) e^{-\frac{\tau}{4}}  \left[ v_y^+(\tau,0^+) - v_y^-(\tau,0^+) \right].
$$
By using Lemmas \ref{lem-Abel} and \ref{lem-Abel-integral} as described above, we obtain the following integral equation:	
\begin{eqnarray*}
	\tilde{\gamma}(t) & = & -\frac{1}{\sqrt{4\pi t}} \int_0^{\infty} 
	\left[ v_0^{+ \prime}(\eta) + v_0^{- \prime}(\eta) \right] 
	\left( \frac{\eta + t}{2 t} \right) e^{-\frac{\eta^2}{4t}} d \eta	\\
	&& 
	- \frac{1}{2} \tilde{\gamma}(t) e^{-\frac{t}{4}} \left[ v_y^+(t,0^+) - v_y^-(t,0^+) \right] 
	- \frac{1}{2}
	\int_0^t \frac{\tilde{\gamma}(\tau) e^{-\frac{\tau}{4}}}{\sqrt{4 \pi (t-\tau)}}  \left[ v_{y}^+(\tau,0^+) - v_{y}^-(\tau,0^+) \right] d \tau \\ 
	&& - \int_0^t \frac{\tilde{\gamma}(\tau) e^{-\frac{\tau}{4}}}{\sqrt{4 \pi (t-\tau)}} \int_0^{\infty} \left[ v_{yy}^+ - v_{yy}^- - \frac{1}{2} v_y^+ + \frac{1}{2} v_y^- \right](\tau,\eta) \left( \frac{\eta + t - \tau}{2 (t-\tau)} \right) e^{-\frac{\eta^2}{4(t-\tau)}} d\eta d \tau.	
\end{eqnarray*}
Substituting the transformation (\ref{v-tilde-v}) yields the integral equation for $\gamma(t)$.
\end{proof}

Next, we solve the integral equations in Lemmas \ref{lemma-gen-2} and \ref{lemma-gen-4}.

The following lemma guarantees existence of the global solutions $u^{\pm} \in C(\mathbb{R}_+,H^2(\mathbb{R}_+) \cap W^{2,\infty}(\mathbb{R}_+))$ 
to the boundary-value problems (\ref{u-eqs}) for small initial data $u_0^{\pm} \in H^2(\mathbb{R}_+) \cap W^{2,\infty}(\mathbb{R}_+)$ and the given function $\gamma \in L^1(\mathbb{R}_+) \cap L^{\infty}(\mathbb{R}_+) \cap C(\mathbb{R}_+)$. The global solutions satisfy the dynamical conditions (\ref{u-xi-dot}) but do not generally satisfy the additional conditions (\ref{u-continuity}) and (\ref{u-redundant-cond}). 

\begin{lemma}
\label{lemma-gen-3}
For every $\epsilon > 0$ (small enough), there is $\delta > 0$ such that
for every $u_0^{\pm} \in H^2(\mathbb{R}_+) \cap W^{2,\infty}(\mathbb{R}_+)$ 
and for every $\gamma \in L^1(\mathbb{R}_+) \cap L^{\infty}(\mathbb{R}_+) \cap C(\mathbb{R}_+)$ satisfying
\begin{equation}
\label{initial-bound-lemma}
\| u_0^+ \|_{H^2 \cap W^{2,\infty}} + \| u_0^- \|_{H^2 \cap W^{2,\infty}} 
+ \| \gamma \|_{L^1\cap L^{\infty}} < \delta,
\end{equation}
there exist the unique solutions $u^{\pm} \in C(\mathbb{R}_+,H^2(\mathbb{R}_+) \cap W^{2,\infty}(\mathbb{R}_+))$
to the integral equations in Lemma \ref{lemma-gen-2}. Moreover, the solutions satisfy
\begin{equation}
\label{stability-gen}
\| u^+(t,\cdot) \|_{H^2 \cap W^{2,\infty}} + \| u^-(t,\cdot) \|_{H^2 \cap W^{2,\infty}} < \epsilon \quad \quad t > 0
\end{equation}
and the dynamical conditions (\ref{u-xi-dot}) for $t > 0$.
\end{lemma}

\begin{proof}
We rewrite the integral equations in Lemma \ref{lemma-gen-2} 
as the fixed-point equations associated with the following integral operators:
\begin{equation}
\label{fixed-point}
u^{\pm} = A^{\pm}(u^{\pm}) := u^{\pm}_1 + u^{\pm}_2 \pm u^{\pm}_3,
\end{equation} 
where
\begin{eqnarray}
\label{part-1}
u^{\pm}_1(t,y) & = & \frac{1}{\sqrt{4 \pi t}} \int_0^{\infty} u_0^{\pm}(\eta) \left[
e^{-\frac{(y-\eta + t)^2}{4t}} - e^{-y} e^{-\frac{(y+\eta - t)^2}{4t}} \right] d \eta, \\
\label{part-2}
u^{\pm}_2(t,y) & = & \int_0^t \frac{\gamma(\tau) d \tau}{\sqrt{4 \pi (t-\tau)}}
\int_0^{\infty} e^{-\eta} \left[
e^{-\frac{(y-\eta + t - \tau)^2}{4(t - \tau)}} - e^{-y} e^{-\frac{(y+\eta - t + \tau)^2}{4 (t-\tau)}} \right] d\eta, \\
\label{part-3}
u^{\pm}_3(t,y) & = & \int_0^t \frac{\gamma(\tau) d \tau}{\sqrt{4 \pi (t-\tau)}}
\int_0^{\infty} u_{\eta}^{\pm}(\tau,\eta) \left[
e^{-\frac{(y-\eta + t - \tau)^2}{4(t - \tau)}} - e^{-y} e^{-\frac{(y+\eta - t + \tau)^2}{4 (t-\tau)}} \right] d\eta.
\end{eqnarray}
The fixed-point equations (\ref{fixed-point}) are considered 
in a small ball $B_{\epsilon} \subset X$ of radius $\epsilon > 0$ in Banach space
$$
X := L^{\infty}(\mathbb{R}_+,H^2(\mathbb{R}_+) \cap W^{2,\infty}(\mathbb{R}_+)),
$$ 
where $u_0^{\pm} \in H^2(\mathbb{R}_+) \cap W^{2,\infty}(\mathbb{R}_+)$ 
and $\gamma \in L^1(\mathbb{R}) \cap L^{\infty}(\mathbb{R})$ are given and satisfy 
the initial bound (\ref{initial-bound-lemma}). We analyze hereafter each term 
in the definition of $A^{\pm}(u^{\pm})$ in $X$.

The explicit expressions for $u_1^{\pm}$ in (\ref{part-1}) coincide with (\ref{sol-explicit}) after the change of the initial data
$u_0$ to $u_0^{\pm}$. By using the same analysis as in the proof of Lemma \ref{lemma-2}, we obtain the same bounds (\ref{bound-derivative-zero}),
(\ref{bound-derivative-first}), and (\ref{bound-derivative-second}) for $u^{\pm}_1(t,\cdot)$ and their first and second
$y$-derivatives in the $L^2(\mathbb{R}_+)$ norm. 
Similarly, the same bounds can be rederived in the $L^{\infty}(\mathbb{R}_+)$ norm. Combining them together, we deduce that there exists $C > 0$ such that 
\begin{equation}
\label{part-one}
\| u_1^{\pm}(t,\cdot) \|_{H^2 \cap W^{2,\infty}} \leq C \| u_0^{\pm} \|_{H^2 \cap W^{2,\infty}}, \quad t > 0.
\end{equation}
By Lebesgue's dominated convergence theorem, 
$u_1^{\pm}(t,\cdot)$ and their first and second derivatives in $y$ 
are continuous functions of $y$ for every $t > 0$ 
such that taking the limit $y \to 0^+$ yields
\begin{equation}
\label{inter-1}
\partial_y u_1^{\pm}(t,0^+) + \partial^2_y u_1^{\pm}(t,0^+) = 0, \quad t > 0.
\end{equation}

Let us now consider the explicit expressions for $u_2^{\pm}$ in (\ref{part-2}). 
By the generalized Young's inequality (\ref{Young}) with either $p = 1$ and 
$q = r = 2$ or $p = q = 2$ and $r = \infty$, we obtain
\begin{eqnarray*}
\| u_2^{\pm}(t,\cdot) \|_{L^2 \cap L^{\infty}} & \leq & 2 \int_0^t |\gamma(\tau)|
\| e^{-y} \ast G(t-\tau,y +t -\tau) \|_{L^2_y(\mathbb{R}_+) \cap L^{\infty}_y(\mathbb{R}_+)} d \tau \\
& \leq & 2 \int_0^t |\gamma(\tau)|
\| e^{-y} \|_{L^1_y(\mathbb{R}_+) \cap L^2_y(\mathbb{R}_+)} \| G(t-\tau,y +t -\tau) \|_{L^2_y(\mathbb{R})} d \tau \\
& \leq & \frac{2}{(8\pi)^{1/4}} \int_0^t \frac{|\gamma(\tau)|}{(t-\tau)^{1/4}} d \tau,
\end{eqnarray*}
where the second equality in (\ref{heat-kernel}) has been used together with 
$\| e^{-y} \|_{L^1_y(\mathbb{R}_+)} = 1$ and $\| e^{-y} \|_{L^2_y(\mathbb{R}_+)} = \frac{1}{\sqrt{2}} < 1$.
Computing derivatives in $y$ and integrating by parts yield
\begin{eqnarray*}
\partial_y u^{\pm}_2(t,y) = -\int_0^t \frac{\gamma(\tau) d \tau}{\sqrt{4 \pi (t-\tau)}}
\int_0^{\infty} e^{-\eta} e^{-\frac{(y-\eta + t - \tau)^2}{4(t - \tau)}} d\eta
+ \nu(t,y)
\end{eqnarray*}
and
\begin{eqnarray*}
	\partial^2_y u^{\pm}_2(t,y) = \int_0^t \frac{\gamma(\tau) d \tau}{\sqrt{4 \pi (t-\tau)}}
	\int_0^{\infty} e^{-\eta} e^{-\frac{(y-\eta + t - \tau)^2}{4(t - \tau)}} d\eta - \frac{1}{2} \nu(t,y) + \nu_y(t,y),
\end{eqnarray*}
where $\nu(t,y)$ is given by (\ref{def-nu}).
By using estimates in the proof of Lemma \ref{prop-tech}, we obtain
\begin{eqnarray*}
\| \partial_y u_2^{\pm}(t,\cdot) \|_{L^2} \leq \frac{3}{(8\pi)^{1/4}} \int_0^t \frac{|\gamma(\tau)|}{(t-\tau)^{1/4}}d \tau,
\end{eqnarray*}
\begin{eqnarray*}
	\| \partial_y u_2^{\pm}(t,\cdot) \|_{L^{\infty}} \leq 	\frac{3}{\sqrt{4 \pi}} \int_0^t \frac{|\gamma(\tau)|}{(t-\tau)^{1/2}} d\tau,
\end{eqnarray*}
\begin{eqnarray*}
\| \partial^2_y u_2^{\pm}(t,\cdot) \|_{L^2} \leq \frac{2}{(8\pi)^{1/4}} \int_0^t \frac{|\gamma(\tau)|}{(t-\tau)^{1/4}}d \tau + \frac{1}{(8 \pi)^{1/4}} \int_0^t \frac{|\gamma(\tau)|}{(t - \tau)^{3/4}} d \tau,
\end{eqnarray*}
and
\begin{eqnarray*}
	\| \partial^2_y u_2^{\pm}(t,\cdot) \|_{L^{\infty}} \leq \frac{3}{\sqrt{4 \pi}} \int_0^t \frac{|\gamma(\tau)|}{(t-\tau)^{1/2}} d\tau + |\gamma(t)|.
\end{eqnarray*}
Combining all estimates together, we deduce that there exists $C > 0$ such that 
\begin{equation}
\label{part-two}
\| u_2^{\pm}(t,\cdot) \|_{H^2 \cap W^{2,\infty}} \leq C \left( \int_0^t \frac{|\gamma(\tau)|}{(t-\tau)^{1/4}}d \tau + |\gamma(t)| \right), 
\quad t > 0,
\end{equation}
where the end point estimates are taken into account. 
Moreover, $u_2^{\pm}(t,\cdot)$ and their first and second derivatives in $y$ 
are continuous functions of $y$ for every $t > 0$.
By taking the limit $y \to 0^+$ and using (\ref{bc-nu}) in Lemma \ref{prop-tech}, we obtain 
\begin{equation}
\label{inter-2}
\partial_y u_2^{\pm}(t,0^+) + \partial^2_y u_2^{\pm}(t,0^+) = 
\frac{1}{2} \nu(t,0) + \nu_y(t,0^+) = - \gamma(t), \quad t > 0.
\end{equation}

We turn now to the explicit expressions for $u_3^{\pm}$ in (\ref{part-3}). 
Integrating by parts, we obtain 
\begin{eqnarray*}
u^{\pm}_3(t,y) &=& -\int_0^t \frac{\gamma(\tau) d \tau}{\sqrt{4 \pi (t-\tau)}}
\int_0^{\infty} u^{\pm}(\tau,\eta) \left(\frac{y-\eta + t - \tau}{2(t-\tau)}\right) 
e^{-\frac{(y-\eta + t - \tau)^2}{4(t - \tau)}} d\eta \\
&& -  \int_0^t \frac{\gamma(\tau) d \tau}{\sqrt{4 \pi (t-\tau)}}
\int_0^{\infty} u^{\pm}(\tau,\eta) \left(
\frac{y+\eta - t + \tau}{2(t-\tau)} \right)
e^{-y} e^{-\frac{(y+\eta - t + \tau)^2}{4 (t-\tau)}} d\eta.
\end{eqnarray*}
By the generalized Young's inequality (\ref{Young}) with $p = 1$ and 
either $q = r = 2$ or $q = r = \infty$, we obtain
\begin{eqnarray*}
	\| u_3^{\pm}(t,\cdot) \|_{L^2 \cap L^{\infty}} & \leq & 2 \int_0^t |\gamma(\tau)|
	\| u^{\pm}(\tau,y) \ast \partial_y G(t-\tau,y +t -\tau) \|_{L^2_y(\mathbb{R}_+)\cap L^{\infty}_y(\mathbb{R}_+)} d \tau \\
	& \leq & 2 \int_0^t |\gamma(\tau)|
	\|  u^{\pm}(\tau,\cdot) \|_{L^2 \cap L^{\infty}} \| \partial_y G(t-\tau,y +t -\tau) \|_{L^1_y(\mathbb{R})} d \tau \\
	& \leq & 2 \int_0^t \frac{|\gamma(\tau)|}{\sqrt{\pi (t-\tau)}} \|  u^{\pm}(\tau,\cdot) \|_{L^2 \cap L^{\infty}} d \tau,
\end{eqnarray*}
where the first equality in (\ref{heat-kernel-derivative}) has been used. 
Computing derivative in $y$ and integrating by parts yield
\begin{eqnarray*}
	\partial_y u^{\pm}_3(t,y) &=& -\int_0^t \frac{\gamma(\tau) d \tau}{\sqrt{4 \pi (t-\tau)}}
	\int_0^{\infty} u_y^{\pm}(\tau,\eta) \left(\frac{y-\eta + t - \tau}{2(t-\tau)}\right) 
	e^{-\frac{(y-\eta + t - \tau)^2}{4(t - \tau)}} d\eta \\
	&& +  \int_0^t \frac{\gamma(\tau) d \tau}{\sqrt{4 \pi (t-\tau)}}
	\int_0^{\infty} u_y^{\pm}(\tau,\eta) \left(
	\frac{y+\eta - t + \tau}{2(t-\tau)} \right)
	e^{-y} e^{-\frac{(y+\eta - t + \tau)^2}{4 (t-\tau)}} d\eta \\
	&& + \int_0^t \frac{\gamma(\tau) d \tau}{\sqrt{4 \pi (t-\tau)}}
\int_0^{\infty} u^{\pm}(\tau,\eta) \left(
\frac{y+\eta - t + \tau}{2(t-\tau)} \right)
e^{-y} e^{-\frac{(y+\eta - t + \tau)^2}{4 (t-\tau)}} d\eta	
\end{eqnarray*}
With similar estimates as above, we obtain 
\begin{eqnarray*}
	\| \partial_y u_3^{\pm}(t,\cdot) \|_{L^2 \cap L^{\infty}} \leq \int_0^t \frac{|\gamma(\tau)|}{\sqrt{\pi (t-\tau)}} (2 \|  \partial_y u^{\pm}(\tau,\cdot) \|_{L^2 \cap L^{\infty}} + \|  u^{\pm}(\tau,\cdot) \|_{L^2 \cap L^{\infty}}) d \tau.
\end{eqnarray*}
Computing another derivative in $y$ and integrating by parts yield
\begin{eqnarray*}
	\partial^2_y u^{\pm}_3(t,y) &=& -\int_0^t \frac{\gamma(\tau) d \tau}{\sqrt{4 \pi (t-\tau)}}
	\int_0^{\infty} u_{yy}^{\pm}(\tau,\eta) \left(\frac{y-\eta + t - \tau}{2(t-\tau)}\right) 
	e^{-\frac{(y-\eta + t - \tau)^2}{4(t - \tau)}} d\eta \\
	&& -  \int_0^t \frac{\gamma(\tau) d \tau}{\sqrt{4 \pi (t-\tau)}}
	\int_0^{\infty} u_{yy}^{\pm}(\tau,\eta) \left(
	\frac{y+\eta - t + \tau}{2(t-\tau)} \right)
	e^{-y} e^{-\frac{(y+\eta - t + \tau)^2}{4 (t-\tau)}} d\eta \\
	&& - 2 \int_0^t \frac{\gamma(\tau) d \tau}{\sqrt{4 \pi (t-\tau)}}
	\int_0^{\infty} u_y^{\pm}(\tau,\eta) \left(
	\frac{y+\eta - t + \tau}{2(t-\tau)} \right)
	e^{-y} e^{-\frac{(y+\eta - t + \tau)^2}{4 (t-\tau)}} d\eta \\
		&& - \int_0^t \frac{\gamma(\tau) d \tau}{\sqrt{4 \pi (t-\tau)}}
	\int_0^{\infty} u^{\pm}(\tau,\eta) \left(
	\frac{y+\eta - t + \tau}{2(t-\tau)} \right)
	e^{-y} e^{-\frac{(y+\eta - t + \tau)^2}{4 (t-\tau)}} d\eta	\\
&& - \int_0^t \frac{\gamma(\tau)}{\sqrt{4 \pi (t-\tau)}}
\partial_y u^{\pm}(\tau,0^+) \left( \frac{y}{t-\tau} \right)
e^{-\frac{(y + t - \tau)^2}{4 (t-\tau)}} d \tau,
\end{eqnarray*}
where the last term can be written as $\tilde{\nu}_y(t,y) + \frac{1}{2} \tilde{\nu}(t,y)$ with 
$$
\tilde{\nu}(t,y) := 2 \int_0^t \frac{\gamma(\tau) u^{\pm}(\tau,0^+)}{\sqrt{4 \pi (t-\tau)}} 
e^{-\frac{(y + t - \tau)^2}{4 (t-\tau)}} d \tau.
$$
All terms in $\partial^2_y u_3^{\pm}$ including the last one are estimated similarly to what was done above. As a result, we obtain
\begin{eqnarray*}
\begin{array}{rcl}
\| \partial^2_y u_3^{\pm}(t,\cdot) \|_{L^2 \cap L^{\infty}}
& \leq & 2  \int_0^t \frac{|\gamma(\tau)| d\tau}{\sqrt{\pi (t-\tau)}} 
(\|  \partial^2_y u^{\pm}(\tau,\cdot) \|_{L^2 \cap L^{\infty}} + 
	\|  \partial_y u^{\pm}(\tau,\cdot) \|_{L^2 \cap L^{\infty}}) d\tau \\
&& + \int_0^t \frac{|\gamma(\tau)| d\tau}{\sqrt{\pi (t-\tau)}}  \|  u^{\pm}(\tau,\cdot) \|_{L^2 \cap L^{\infty}} d\tau 
+ \| \tilde{\nu}_y(t,\cdot) + \frac{1}{2} \tilde{\nu}(t,\cdot) \|_{L^2 \cap L^{\infty}},
\end{array}
\end{eqnarray*}
where the following estimates from the proof of Lemma \ref{prop-tech} can be used:
$$
\| \tilde{\nu}_y(t,\cdot) + \frac{1}{2} \tilde{\nu}(t,\cdot)  \|_{L^2} \leq \frac{1}{(8\pi)^{1/4}} 
\int_0^t \frac{|\gamma(\tau)| | \partial_y u^{\pm}(\tau,0^+)|}{(t-\tau)^{1/4}}
d\tau + \frac{1}{(8\pi)^{1/4}} 
\int_0^t \frac{|\gamma(\tau)| | \partial_y u^{\pm}(\tau,0^+)|}{(t-\tau)^{3/4}}
d\tau 
$$
and
$$
\| \tilde{\nu}_y(t,\cdot) + \frac{1}{2} \tilde{\nu}(t,\cdot)  \|_{L^{\infty}} \leq |\gamma(\tau)| | \partial_y u^{\pm}(\tau,0^+)|.
$$
Combining all estimates together, we deduce that there exists $C > 0$ such that 
\begin{eqnarray}
\nonumber
\| u_3^{\pm}(t,\cdot) \|_{H^2\cap W^{2,\infty}} & \leq & C \left( \int_0^t \frac{|\gamma(\tau)|}{(t-\tau)^{1/2}} 
\| u^{\pm}(\tau,\cdot) \|_{H^2 \cap W^{2,\infty}}
d\tau +
\int_0^t \frac{|\gamma(\tau)| | \partial_y u^{\pm}(\tau,0^+)|}{(t-\tau)^{1/4}}
d\tau \right. \\
&& \qquad \qquad \left. +  |\gamma(\tau)| | \partial_y u^{\pm}(\tau,0^+)| \right), \quad t > 0,
\label{part-three}
\end{eqnarray} 
where the end point estimates are taken into the account.
Moreover, $u_3^{\pm}(t,\cdot)$ and their first and second derivatives in $y$ 
are continuous functions of $y$ for every $t > 0$.
By taking the limit $y \to 0^+$, we obtain 
\begin{equation}
\label{inter-3}
\partial_y u_3^{\pm}(t,0^+) + \partial^2_y u_3^{\pm}(t,0^+) = 
- \gamma(t) \partial_y u^{\pm}(t,0^+), \quad t > 0.
\end{equation}
Summing (\ref{inter-1}), (\ref{inter-2}), and (\ref{inter-3}) recovers the dynamical conditions (\ref{u-xi-dot}) for $u^{\pm}$.

Next, we run the fixed-point arguments for the fixed-point equations 
(\ref{fixed-point}) in $B_{\epsilon} \subset X$. 
If $u^{\pm}_0$ and $\gamma$ satisfy the initial bound (\ref{initial-bound-lemma}), then there exists $C > 0$ such that 
$$
\| A^{\pm}(0) \|_X \leq C \delta
$$ 
due to bounds  (\ref{part-one}) and (\ref{part-two}), where we have also used
the bound (\ref{bound-gamma}) in Lemma \ref{lem-tech}. Furthermore, 
for every small $\epsilon > 0$, there is sufficiently small $\delta > 0$ 
such that if $u^{\pm} \in B_{\epsilon} \subset X$, then $A^{\pm}(u^{\pm}) \in B_{\epsilon} \subset X$; 
moreover $A^{\pm}$ are contractions on $B_{\epsilon} \subset X$ due to 
bounds (\ref{part-three}), where the bound (\ref{bound-gamma}) can be used again. Existence and uniqueness of the fixed point  $u^{\pm} \in B_{\epsilon} \subset X$
to the fixed-point equations (\ref{fixed-point}) follows by the Banach fixed-point theorem. Hence, the bound (\ref{stability-gen}) is proven. 
By the standard bootstrapping arguments, if $\gamma \in C(\mathbb{R}_+)$, 
we also get 
$$
u^{\pm} \in C(\mathbb{R}_+,H^2(\mathbb{R}_+)\cap W^{2,\infty}(\mathbb{R}_+)).
$$
The proof of the lemma is complete.
\end{proof}

When $u^{\pm} \in C(\mathbb{R}_+,H^2(\mathbb{R}_+) \cap W^{2,\infty}(\mathbb{R}_+))$  are substituted from Lemma \ref{lemma-gen-3} into the integral equation (\ref{eq-gamma}), we are looking for a small 
solution $\gamma \in L^1(\mathbb{R}_+) \cap L^{\infty}(\mathbb{R}_+) \cap C(\mathbb{R}_+)$ in response to small initial data $u_0^{\pm} \in H^2(\mathbb{R}_+) \cap W^{2,\infty}(\mathbb{R}_+)$. However, 
we were not able to close the fixed-point iterations unless we added 
the additional requirement of the spatial exponential decay 
of the initial data $u_0^{\pm}$. 

The following lemma shows that the spatial exponential decay of the initial data $u_0^{\pm}$ is preserved in time.

\begin{lemma}
	\label{lemma-gen-3-prime}
In addition to (\ref{initial-bound-lemma}),
we assume that $u_0^{\pm} \in H^2(\mathbb{R}_+) \cap W^{2,\infty}(\mathbb{R}_+)$ satisfy
	\begin{equation}
	\label{initial-bound-lemma-prime}
	\| e^{\alpha \cdot}  u_0^+ \|_{W^{2,\infty}} + 
	\| e^{\alpha \cdot} u_0^- \|_{W^{2,\infty}} < \delta,
	\end{equation}
	for a fixed $\alpha \in (0,\frac{1}{2}]$. Then, the unique solutions $u^{\pm} \in C(\mathbb{R}_+,H^2(\mathbb{R}_+) \cap W^{2,\infty}(\mathbb{R}_+))$ of Lemma \ref{lemma-gen-3} satisfy 
	\begin{equation}
	\label{stability-gen-prime}
	\| e^{\alpha \cdot}  u^+(t,\cdot) \|_{W^{2,\infty}} + \| e^{\alpha \cdot} u^-(t,\cdot) \|_{W^{2,\infty}} < \epsilon, \quad \quad t > 0
	\end{equation}
	and
	\begin{equation}
	\label{asym-stability-decay-gen}
	\| u^{\pm}(t,\cdot) \|_{W^{2,\infty}} \to 0 \quad \mbox{\rm as} \quad t \to +\infty.
	\end{equation} 
\end{lemma}

\begin{proof}
	By rearranging the heat kernels, we can rewrite (\ref{part-1}), (\ref{part-2}), and (\ref{part-3}) as
	\begin{eqnarray}
\label{part-1-prime}
e^{\alpha y} u^{\pm}_1(t,y) = \frac{e^{-\alpha(1-\alpha) t}}{\sqrt{4 \pi t}} \int_0^{\infty} e^{\alpha \eta} u_0^{\pm}(\eta) \left[
	e^{-\frac{(y-\eta + (1-2\alpha) t)^2}{4t}} - e^{-(1-2\alpha)y} e^{-\frac{(y+\eta - (1-2\alpha)t)^2}{4t}} \right] d \eta \; , 
		\end{eqnarray}
		\begin{eqnarray}
\nonumber
&& 	e^{\alpha y} u^{\pm}_2(t,y) = \int_0^t \frac{\gamma(\tau) e^{-\alpha(1-\alpha) (t-\tau)} d \tau}{\sqrt{4 \pi (t-\tau)}} \qquad \qquad \qquad  \qquad \qquad  \qquad \qquad \qquad  \qquad \qquad  \qquad \qquad  \\
		\label{part-2-prime}
	&& \qquad \qquad \qquad 
\times	\int_0^{\infty} e^{-(1-\alpha)\eta} \left[
	e^{-\frac{(y-\eta + (1-2\alpha) (t - \tau))^2}{4(t - \tau)}} - e^{-(1-2\alpha)y} e^{-\frac{(y+\eta - (1-2\alpha)(t - \tau))^2}{4 (t-\tau)}} \right] d\eta,	\end{eqnarray}
	\begin{eqnarray}
	\nonumber
&&	e^{\alpha y} u^{\pm}_3(t,y) = \int_0^t \frac{\gamma(\tau) e^{-\alpha(1-\alpha) (t-\tau)} d \tau}{\sqrt{4 \pi (t-\tau)}} \qquad \qquad \qquad  \qquad \qquad  \qquad \qquad \qquad  \qquad \qquad  \qquad \qquad \\
		\label{part-3-prime}
		&& \qquad \qquad \qquad \times	\int_0^{\infty} e^{\alpha \eta} u_{\eta}^{\pm}(\tau,\eta) \left[
	e^{-\frac{(y-\eta + (1-2\alpha)(t - \tau))^2}{4(t - \tau)}} - e^{-(1-2\alpha)y} e^{-\frac{(y+\eta - (1-2\alpha)(t- \tau))^2}{4 (t-\tau)}} \right] d\eta.
	\end{eqnarray}
If $\alpha \in (0,\frac{1}{2}]$, the exponential function $e^{-(1-2\alpha) y}$ is still bounded on $\mathbb{R}_+$, whereas $e^{\alpha y} u_0^{\pm}(y)$  belongs to $W^{2,\infty}(\mathbb{R}_+)$ and satisfies the initial bound (\ref{initial-bound-lemma-prime}). 
All convolution estimates of Lemma \ref{lemma-gen-3} hold true with some $\alpha$-dependent constants and give 
the unique solution in $W^{2,\infty}(\mathbb{R}_+)$ satisfying the bound (\ref{stability-gen-prime}). 

It remains to prove the asymptotic decay (\ref{asym-stability-decay-gen}). 
Since 
\begin{eqnarray*}
	\| u \|_{L^{\infty}(\mathbb{R}_+)} & \leq & \| e^{\alpha \cdot} u \|_{L^{\infty}(\mathbb{R}_+)}, \\
		\| u_y \|_{L^{\infty}(\mathbb{R}_+)} & \leq & \| (e^{\alpha \cdot} u)_y \|_{L^{\infty}(\mathbb{R}_+)} + \alpha \| e^{\alpha \cdot} u \|_{L^{\infty}(\mathbb{R}_+)}, \\
				\| u_{yy} \|_{L^{\infty}(\mathbb{R}_+)} & \leq & \| (e^{\alpha \cdot} u)_{yy} \|_{L^{\infty}(\mathbb{R}_+)} + 2 \alpha \| (e^{\alpha \cdot} u)_y\|_{L^{\infty}(\mathbb{R}_+)} + \alpha^2 \| e^{\alpha \cdot} u \|_{L^{\infty}(\mathbb{R}_+)}, 
	\end{eqnarray*}
it is sufficient to prove the decay to zero for $e^{\alpha y} u^{\pm}(t,y)$ as $t \to +\infty$ in $W^{2,\infty}(\mathbb{R}_+)$. In order to prove the decay in time, we show henceforth that $e^{\alpha y} u^{\pm}(t,y)$ are bounded in $L^1(\mathbb{R}_+,W^{2,\infty}(\mathbb{R}_+))$ since $\| e^{\alpha \cdot} u^{\pm}(t,\cdot) \|_{W^{2,\infty}}$ are continuous functions of $t \in \mathbb{R}_+$.

Thanks to the decaying exponential function $e^{-\alpha(1 - \alpha)t}$ as $t \to +\infty$ in (\ref{part-1-prime}) and the Young's inequality (\ref{Young}) with $p = r = \infty$ and $q = 1$, there exists the $\alpha$-dependent $C_{\alpha} > 0$ such that 
\begin{equation}
\label{decay-part-one}
\int_0^{\infty} \|e^{\alpha \cdot} u_1^{\pm}(t,\cdot) \|_{W^{2,\infty}} dt \leq C_{\alpha} \| e^{\alpha \cdot} u_0^{\pm} \|_{W^{2,\infty}}.
\end{equation}
Similarly, it follows from (\ref{part-2-prime}) that
\begin{equation}
\label{decay-part-two}
\int_0^{\infty} \|e^{\alpha \cdot} u_2^{\pm}(t,\cdot) \|_{W^{2,\infty}} dt  \leq C_{\alpha} \int_0^{\infty} \int_0^t |\gamma(\tau)| e^{-\alpha(1-\alpha) (t-\tau)} d \tau dt \leq \frac{C_{\alpha}}{\alpha (1-\alpha)} \| \gamma \|_{L^1(\mathbb{R}_+)},
\end{equation}
where we have used the Young's inequality (\ref{Young-time}) with $p = q = r = 1$. 

Finally, integrating (\ref{part-3-prime}) by parts yields
{\small
\begin{eqnarray*}
&&	e^{\alpha y} u^{\pm}_3(t,y) = -\int_0^t \frac{\gamma(\tau) e^{-\alpha(1-\alpha) (t-\tau)} d \tau}{\sqrt{4 \pi (t-\tau)}}
	\int_0^{\infty} e^{\alpha \eta} u^{\pm}(\tau,\eta) \left(\frac{y-\eta + 
		(1-2\alpha)(t - \tau)}{2(t-\tau)}\right) 
	e^{-\frac{(y-\eta + t - \tau)^2}{4(t - \tau)}} d\eta \\
	&& -  \int_0^t \frac{\gamma(\tau) e^{-\alpha(1-\alpha) (t-\tau)} d \tau}{\sqrt{4 \pi (t-\tau)}}
	\int_0^{\infty} e^{\alpha \eta} u^{\pm}(\tau,\eta) \left(
	\frac{y+\eta - (1-2\alpha)(t -\tau)}{2(t-\tau)} \right)
	e^{-(1-2\alpha) y} e^{-\frac{(y+\eta - t + \tau)^2}{4 (t-\tau)}} d\eta.
\end{eqnarray*}
}By using the bounds  (\ref{heat-kernel}) and (\ref{heat-kernel-derivative}),
the Young's inequality (\ref{Young}) with $p = r = \infty$ and $q = 1$, and the Young's inequality (\ref{Young-time}) with $p = q = r = 1$,
we obtain
\begin{equation}
\label{decay-part-three}
\int_0^{\infty} \|e^{\alpha \cdot} u_3^{\pm}(t,\cdot) \|_{W^{2,\infty}} dt  \leq C_{\alpha} \| \gamma \|_{L^1(\mathbb{R}_+)} \sup_{t \in \mathbb{R}_+} \| e^{\alpha \cdot} u^{\pm}(t,\cdot) \|_{W^{2,\infty}},
\end{equation}
for every solutions $u^{\pm} \in C(\mathbb{R}_+,H^2(\mathbb{R}_+) \cap W^{2,\infty}(\mathbb{R}_+))$ of Lemma \ref{lemma-gen-3} satisfying (\ref{stability-gen-prime}). Hence, $e^{\alpha y} u^{\pm}(t,y)$ are bounded in $L^1(\mathbb{R}_+,W^{2,\infty}(\mathbb{R}_+))$, 
which implies the asymptotic decay (\ref{asym-stability-decay-gen}) 
since $\| e^{\alpha \cdot} u^{\pm}(t,\cdot) \|_{W^{2,\infty}}$ are continuous functions of $t \in \mathbb{R}_+$.
\end{proof}

\begin{remark}
Due to the exponential decay with $\alpha \in (0,\frac{1}{2}]$, we also have the bound 
$$
\| u \|_{H^2(\mathbb{R}_+)} \leq C_{\alpha} \| e^{\alpha \cdot} u \|_{W^{2,\infty}(\mathbb{R}_+)},
$$
which implies that $\| u^{\pm}(t,\cdot) \|_{H^2} \to 0$ as $t \to \infty$.
This decay in time is impossible if the initial data do not satisfy the 
spatial exponential decay, see Remark \ref{remark-decay-H2}.
\end{remark}

The final lemma gives the existence of the unique solution  
to the integral equation (\ref{eq-gamma}) 
for $\gamma \in L^1(\mathbb{R}_+) \cap L^{\infty}(\mathbb{R}_+) \cap C(\mathbb{R}_+)$, 
where $u^{\pm} \in C(\mathbb{R}_+,H^2(\mathbb{R}_+) \cap W^{2,\infty}(\mathbb{R}_+))$  are substituted from Lemmas \ref{lemma-gen-3} and \ref{lemma-gen-3-prime} into the integral equation (\ref{eq-gamma}) and the initial data $u_0^{\pm} \in H^2(\mathbb{R}_+) \cap W^{2,\infty}(\mathbb{R}_+)$ satisfy the bounds (\ref{initial-bound-lemma}) and (\ref{initial-bound-lemma-prime}).

\begin{lemma}
	\label{lemma-gen-5}
Fix $\alpha \in (0,\frac{1}{2}]$ and consider the integral equation (\ref{eq-gamma}) with 
the unique solutions $u^{\pm} \in C(\mathbb{R}_+,H^2(\mathbb{R}_+) \cap W^{2,\infty}(\mathbb{R}_+))$ defined in Lemmas \ref{lemma-gen-3} and \ref{lemma-gen-3-prime} that depend on (small) $\gamma \in L^1(\mathbb{R}_+) \cap L^{\infty}(\mathbb{R}_+) \cap C(\mathbb{R}_+) $. For every $\tilde{\epsilon} > 0$ (small enough), there is $\tilde{\delta} > 0$  such that
	for every $u_0^{\pm} \in H^2(\mathbb{R}_+ \cap W^{2,\infty}(\mathbb{R}_+))$ satisfying 
	\begin{equation}
	\label{initial-bound-lemma-final}
\| u_0^+ \|_{H^2 \cap W^{2,\infty}} + \| u_0^- \|_{H^2 \cap W^{2,\infty}} 
+ \| e^{\alpha \cdot} u_0^+ \|_{W^{2,\infty}} + \| e^{\alpha \cdot} u_0^- \|_{W^{2,\infty}} \leq \tilde{\delta}
	\end{equation}
and the continuity condition $u_0^{+ \prime}(0^+) + u_0^{-\prime}(0^+) = 0$,
there exists the unique solution $\gamma \in L^1(\mathbb{R}_+) \cap L^{\infty}(\mathbb{R}_+) \cap C(\mathbb{R}_+)$ of the integral equation (\ref{eq-gamma})
	satisfying 
	\begin{equation}
	\label{gamma-final-bound}
		\| \gamma \|_{L^{\infty} \cap L^1} \leq \tilde{\epsilon}.
	\end{equation}
\end{lemma}

\begin{proof}
We rewrite the integral equation in Lemma \ref{lemma-gen-4} as 
the fixed-point equation associated with the following integral integral operator:
\begin{equation}
\label{fixed-point-gamma}
\gamma = \mathcal{A}(\gamma) := \gamma_1 + \gamma_2 + \gamma_3,
\end{equation} 
where
\begin{eqnarray*}
&& \gamma_1(t) = -\frac{1}{\sqrt{4\pi t}} \int_0^{\infty} 
\left[ u_0^{+ \prime}(\eta) + u_0^{- \prime}(\eta) + \frac{1}{2} u_0^+(\eta) + \frac{1}{2} u_0^-(\eta) \right] 
\left( \frac{\eta + t}{2 t} \right) e^{-\frac{(\eta-t)^2}{4t}} d \eta,	 \\
&& \gamma_2(t) = 
- \frac{1}{2} \gamma(t) \left[ u_y^+(t,0^+) - u_y^-(t,0^+) \right] 
- \frac{1}{2}
\int_0^t \frac{\gamma(\tau) e^{-\frac{t -\tau}{4}}}{\sqrt{4 \pi (t-\tau)}}  \left[ u_{y}^+(\tau,0^+) - u_{y}^-(\tau,0^+) \right] d \tau, \\ 
&& \gamma_3(t) =  - \int_0^t \frac{\gamma(\tau)}{\sqrt{4 \pi (t-\tau)}} \int_0^{\infty} \left[ u_{yy}^+ - u_{yy}^- + \frac{1}{2} u_y^+ - \frac{1}{2} u_y^- \right](\tau,\eta) \left( \frac{\eta + t - \tau}{2 (t-\tau)} \right) e^{-\frac{(\eta - t + \tau)^2}{4(t-\tau)}} d\eta d \tau.
\end{eqnarray*}
The fixed-point equation (\ref{fixed-point-gamma}) is considered in a small ball $B_{\tilde{\epsilon}} \subset L^1(\mathbb{R}_+) \cap L^{\infty}(\mathbb{R}_+)$ 
of radius $\tilde{\epsilon} > 0$, where $u_0^{\pm} \in H^2(\mathbb{R}_+) \cap W^{2,\infty}(\mathbb{R}_+)$ are given and satisfy (\ref{initial-bound-lemma-final}) 
and $u^{\pm} \in C(\mathbb{R}_+,H^2(\mathbb{R}_+) \cap W^{2,\infty}(\mathbb{R}_+))$ are defined in Lemmas \ref{lemma-gen-3} and \ref{lemma-gen-3-prime} such that $\tilde{\delta}$ and $\tilde{\epsilon}$ in (\ref{initial-bound-lemma-final}) and (\ref{gamma-final-bound}) are smaller than $\delta$ in (\ref{initial-bound-lemma}).
We analyze hereafter each term in the definition of $\mathcal{A}(\gamma)$ in $L^1(\mathbb{R}_+) \cap L^{\infty}(\mathbb{R}_+)$.

Since the initial constraint  $u_0^{+ \prime}(0^+) + u_0^{-\prime}(0^+) = 0$  is satisfied and $u_0^{\pm} \in W^{2,\infty}(\mathbb{R}_+)$, 
we can use the equivalent form (\ref{sol-Abel-equivalent}) 
in Lemma \ref{lem-Abel} and rewrite $\gamma_1$ in (\ref{fixed-point-gamma}) in the form:
\begin{equation}
\label{gamma-1-new}
\gamma_1(t) = -\frac{1}{\sqrt{4\pi t}} \int_0^{\infty} 
f(\eta) e^{-\frac{(\eta-t)^2}{4t}} d \eta,
\end{equation}
where 
$$
f(\eta) = u_0^{+ \prime \prime}(\eta) + u_0^{- \prime \prime}(\eta) + 
\frac{3}{2} u_0^{+ \prime}(\eta) + \frac{3}{2} u_0^{- \prime}(\eta) 
+ \frac{1}{2} u_0^+(\eta) + \frac{1}{2} u_0^-(\eta).
$$
It follows from the first identity in (\ref{heat-kernel}) that 
there exists $C > 0$ such that
\begin{equation}
\label{part-one-gamma}
\sup_{t \geq 0} |\gamma_1(t)| \leq C
\left( \| u_0^+ \|_{W^{2,\infty}} + \| u_0^- \|_{W^{2,\infty}} \right).
\end{equation}
However, there is no bound on $\| \gamma_1 \|_{L^1}$ unless we add the exponential weight on the initial conditions $u_0^{\pm}$ and rewrite 
$\gamma_1$ in (\ref{gamma-1-new}) in the form:
\begin{equation}
\label{gamma-1-new-new}
\gamma_1(t) = -\frac{e^{-\alpha(1-\alpha) t}}{\sqrt{4\pi t}} \int_0^{\infty} 
e^{\alpha \eta} f(\eta) e^{-\frac{(\eta - (1-2\alpha) t)^2}{4t}} d \eta.
\end{equation}
Now, thanks to the exponential factor $e^{-\alpha(1-\alpha) t}$ decaying to zero as $t \to +\infty$, we obtain 
\begin{eqnarray*}
\| \gamma_1 \|_{L^1} \leq \frac{1}{\alpha (1-\alpha)} \| e^{\alpha \cdot} f \|_{L^{\infty}},
\end{eqnarray*}
so there exists a positive $\alpha$-dependent constant $C_{\alpha}$ such that 
\begin{equation}
\label{part-one-gamma-new}
\| \gamma_1 \|_{L^1} \leq C_{\alpha}
\left( \| e^{\alpha \cdot} u_0^+ \|_{W^{2,\infty}} + \| e^{\alpha \cdot} u_0^- \|_{W^{2,\infty}} \right).
\end{equation}

For $\gamma_2$ in (\ref{fixed-point-gamma}), we obtain 
\begin{equation}
\label{part-two-gamma}
\| \gamma_2 \|_{L^1 \cap L^{\infty}} \leq \frac{1}{2} 
\left( 1 + \int_0^{\infty} \frac{e^{-\frac{t}{4}}}{\sqrt{4\pi t}} dt \right)
\| \gamma \|_{L^1 \cap L^{\infty}} 
\sup_{t \in \mathbb{R}_+} \left( \| u^+(t,\cdot) \|_{W^{1,\infty}} + \| u^-(t,\cdot) \|_{W^{1,\infty}} \right),
\end{equation}
where the expression in brackets is a finite constant. No exponential weight is needed to estimate $\gamma_2$ in $L^1(\mathbb{R}_+) \cap L^{\infty}(\mathbb{R}_+)$.

For $\gamma_3$ in (\ref{fixed-point-gamma}), we use the Young's inequality (\ref{Young}) 
with $p = q = 2$ and $r = \infty$ and obtain
\begin{eqnarray*}
|\gamma_3(t)| & \leq & C 
\left( \int_0^{t} \frac{|\gamma(\tau)| d \tau}{(t-\tau)^{1/4}} + \int_0^t \frac{|\gamma(\tau)| d \tau}{(t-\tau)^{3/4}} \right) \\
&& \qquad \times \sup_{t \in \mathbb{R}_+} \left( \| u^+(t,\cdot) \|_{H^2} + \| u^-(t,\cdot) \|_{H^2} \right), \quad t > 0,
\end{eqnarray*}
which is bounded due to (\ref{bound-gamma}) if $\gamma \in L^1(\mathbb{R}_+) \cap L^{\infty}(\mathbb{R}_+)$. There is no bound on $\| \gamma_3 \|_{L^1}$, unless we add the exponential weight and  
rewrite $\gamma_3$ in the equivalent form:
\begin{eqnarray*}
\gamma_3(t) =  - \int_0^t \frac{\gamma(\tau) e^{-\alpha (1-\alpha)(t-\tau)}}{\sqrt{4 \pi (t-\tau)}} \int_0^{\infty} e^{\alpha \eta} g(\tau,\eta) \left( \frac{\eta + t - \tau}{2 (t-\tau)} \right) e^{-\frac{(\eta - (1-2\alpha)(t-\tau))^2}{4(t-\tau)}} d\eta d \tau,
\end{eqnarray*}
where 
$$
g(\tau,\eta) = u_{yy}^+(\tau,\eta)  - u_{yy}^-(\tau,\eta)  + \frac{1}{2} u_y^+(\tau,\eta)  - \frac{1}{2} u_y^-(\tau,\eta).
$$
By using the Young's inequality (\ref{Young}) 
with $p = r = \infty$ and $q = 1$ and by using the Young's inequality (\ref{Young-time}) with either $p = r = 1$ or $p = r = \infty$ and $q = 1$, we now obtain 
\begin{equation*}
\| \gamma_3 \|_{L^1 \cap L^{\infty}} \leq 
\left( \frac{1}{\alpha} + \int_0^{\infty} \frac{e^{-\alpha(1-\alpha)t}}{\sqrt{\pi t}} dt \right)
\| \gamma \|_{L^1 \cap L^{\infty}} 
\sup_{t \in \mathbb{R}_+} \| e^{\alpha \cdot} g(t,\cdot) \|_{L^{\infty}},
\end{equation*}
so there exists a positive $\alpha$-dependent constant $C_{\alpha}$ such that 
\begin{equation}
\label{part-three-gamma-new}
\| \gamma_3 \|_{L^1 \cap L^{\infty}} \leq 
C_{\alpha} \| \gamma \|_{L^1 \cap L^{\infty}} 
\sup_{t \in \mathbb{R}_+} \left( \| e^{\alpha \cdot} u^+(t,\cdot) \|_{W^{2,\infty}} + \| e^{\alpha \cdot} u^-(t,\cdot) \|_{W^{2,\infty}} \right).
\end{equation}

Next, we run the fixed-point arguments for the fixed-point equation
(\ref{fixed-point-gamma}) in $B_{\tilde{\epsilon}} \subset L^1(\mathbb{R}_+) \cap L^{\infty}(\mathbb{R}_+)$. If $u^{\pm}_0$ satisfy the initial bound (\ref{initial-bound-lemma-final}) and $\gamma \in B_{\tilde{\epsilon}}$, then 
the solutions $u^{\pm} \in C(\mathbb{R}_+,H^2(\mathbb{R}_+) \cap W^{2,\infty}(\mathbb{R}_+))$ in Lemmas \ref{lemma-gen-3} and \ref{lemma-gen-3-prime} satisfy the bounds (\ref{stability-gen}) and (\ref{stability-gen-prime}) if $\tilde{\delta} \leq \delta$ and $\tilde{\epsilon} \leq \delta$. The bounds 
(\ref{part-one-gamma}), (\ref{part-one-gamma-new}), 
(\ref{part-two-gamma}), and (\ref{part-three-gamma-new}) imply 
that $\mathcal{A}(\gamma) \in B_{\tilde{\epsilon}}$ for sufficiently small $\tilde{\delta}$ and given small $\tilde{\epsilon}$. Moreover, $\mathcal{A}$ is a contraction on $B_{\tilde{\epsilon}} \subset L^1(\mathbb{R}_+) \cap L^{\infty}(\mathbb{R}_+)$ due to 
the same bounds (\ref{part-two-gamma}), and (\ref{part-three-gamma-new}) and 
the smallness of the solutions $u^{\pm} \in C(\mathbb{R}_+,H^2(\mathbb{R}_+) \cap W^{2,\infty}(\mathbb{R}_+))$.

Existence and 
uniqueness of the fixed point  $\gamma \in B_{\tilde{\epsilon}} \subset L^1(\mathbb{R}_+) \cap L^{\infty}(\mathbb{R}_+)$
to the fixed-point equation (\ref{fixed-point-gamma}) follows from the Banach fixed-point theorem. Hence, the bound (\ref{gamma-final-bound}) is proven. 
By the standard bootstrapping arguments, if $u^{\pm} \in C(\mathbb{R}_+,H^2(\mathbb{R}_+\cap W^{2,\infty}(\mathbb{R}_+))$ 
and $e^{\alpha y} u^{\pm} \in C(\mathbb{R}_+,W^{2,\infty}(\mathbb{R}_+))$, 
then $\gamma \in C(\mathbb{R}_+)$. The proof of the lemma is complete.
\end{proof}

\begin{proof1}{\em of Theorem \ref{theorem-2}.}
	The existence, uniqueness, and continuous dependence of the solutions $u^{\pm}$ to the boundary-value problems (\ref{u-eqs}) with (\ref{u-continuity}) and (\ref{u-xi-dot}) is obtained from Lemmas \ref{lemma-gen-3}, \ref{lemma-gen-3-prime}, and \ref{lemma-gen-5}
	as follows. For a fixed $\epsilon$ in (\ref{stability-gen}) and (\ref{stability-gen-prime}), there exists a small $\delta$ in (\ref{initial-bound-lemma}) and (\ref{initial-bound-lemma-prime}), 
	for which we select $\tilde{\epsilon}$ in (\ref{gamma-final-bound}) 
	such that $\tilde{\epsilon} \leq \delta$. By Lemma \ref{lemma-gen-5}, 
	there exists $\tilde{\delta}$ in (\ref{initial-bound-lemma-final}) 
	and, if necessary,  we reduce $\tilde{\delta}$ so that $\tilde{\delta} \leq \delta$. Then, the results of Lemmas \ref{lemma-gen-3}, \ref{lemma-gen-3-prime}, and \ref{lemma-gen-5} hold simultaneously for the initial conditions 	satisfying (\ref{initial-bound-lemma-final}), which is obtained from (\ref{initial-2}) by the transformations 
	(\ref{decomposition-interface}) and (\ref{variables-plus-minus}). The bound (\ref{final-2}) follows from 
	$u^{\pm} \in B_{\epsilon} \subset X$ in the proof of Lemma \ref{lemma-gen-3} and the transformations 
	(\ref{decomposition-interface}) and (\ref{variables-plus-minus}).
	The decay (\ref{scattering-2}) follows from the decay (\ref{asym-stability-decay-gen}).
	By Lemmas \ref{lemma-gen-3} and \ref{lemma-gen-3-prime}, the solutions 
	belong to the spaces (\ref{space-2}) and (\ref{space-2-exp}). 
	
	The interface condition (\ref{u-redundant-cond}) follows 
	from (\ref{u-continuity}) and (\ref{u-xi-dot}).
The interface condition (\ref{dynamics-general}) of Lemma \ref{lem-interface} 
follows from the transformation (\ref{variables-plus-minus}) and 
the dynamical condition (\ref{u-xi-dot}). 
The positivity condition (\ref{w-positivity-broken}) follows from the decomposition (\ref{decomposition-interface}) and smallness of $u$ in $W^{1,\infty}(\mathbb{R})$ similarly to the proof of 
Theorem \ref{theorem-1}. 
\end{proof1}

\section{Numerical simulations}
\label{section-numerics}

Here we simulate numerically the boundary-value problem (\ref{Burgers-general}) completed 
with the dynamical equation  (\ref{dynamics-general}) and the interface condition (\ref{interface-general}). The interface location $\xi(t)$ satisfies $\xi(0) = 0$. We define again $\gamma(t) = \xi'(t)$ and 
use $W_0'(y) = e^{-|y|}$. By using new variables
\begin{equation}
\label{change var}
v^{\pm}(t,y)= u(t,y) \mp u(t,-y), \quad y > 0,
\end{equation}
we can rewrite the boundary-value problem (\ref{Burgers-general}) as a system of two coupled equations:
\begin{equation}\label{v-eqs}
\left\{ \begin{array}{ll}
v^{+}_t = v^{+}_y + v_{yy}^{+} + \gamma v_{y}^{-}, \quad & y > 0, \\
v_{t}^{-} = v^{-}_y +v_{yy}^{-} + \gamma v_{y}^{+} + 2 \gamma e^{-y}, \quad & y > 0,\\
\end{array} \right.
\end{equation}
subject to the boundary conditions
\begin{equation}\label{v-bc}
\begin{cases}
v^{\pm}(t,0) = 0,\\
v_{y}^{-}(t,0) =0, \\
v^{\pm}(t,y)\to 0 \textrm{ as } y\to \infty,
\end{cases}
\end{equation}
the interface condition
\begin{equation}
\label{redundant-cond}
v_y^+(t,0) + v_{yy}^{+}(t,0) = 0,
\end{equation}
and the dynamical condition
\begin{equation}
\label{xi-dot}
\gamma(t) = -\frac{v_{yy}^{-}(t,0)}{2 + v_{y}^{+}(t,0)}.
\end{equation}
If $v^-(0,y) = 0$ initially, then $\gamma(t) = 0$ and $v^-(t,y) = 0$ are preserved in the time evolution of (\ref{v-eqs}), (\ref{v-bc}), and (\ref{xi-dot}). In this case, the variable $v^+(t,y)$ satisfies 
the boundary-value problem (\ref{Burgers-odd}), which is analyzed in Theorem \ref{theorem-1} for the odd perturbations to the viscous shock. In what follows, we consider the general case of $v^-(0,y) \neq 0$ which is analyzed in Theorem \ref{theorem-2}.

The spatial domain of system (\ref{v-eqs}) is discretized at the points $y_n = nh$ with equal step size $h$ for $n = 1, \dots, N$.
It follows from the boundary conditions (\ref{v-bc}) that $v^{\pm}(t,y_0) = 0$ at $y_0 = 0$. 
Although the problem is unbounded in one direction, one can truncate the half-line on the finite interval $[0,L]$ with sufficiently large $L$
and $y_{N+1} = L = (N+1)h$ and apply the Dirichlet condition $v^{\pm}(t,y_{N+1}) = 0$ at the end point.
This approach of truncation is commonly adopted for the numerical approximation of evanescent waves in engineering \cite{Numerics} 
as the Dirichlet condition does not provide large errors due to reflections if the waves have fast spatial decay.

At each time level $t_k = k \tau$ with the time step $\tau$, we approximate the spatial derivatives with the second-order central differences as follows:
\begin{align}
v^{\pm}_y(t_k,y_n)  &=  \frac{v^{\pm}_{n+1, k}- v^{\pm}_{n-1, k}}{2h},\\
v^{\pm}_{yy}(t_k,y_n) &=  \frac{v^{\pm}_{n+1, k}-2v^{\pm}_{n,k}+v^{\pm}_{n-1, k}}{h^2}.
\end{align}
where $v_{n,k}$ is a numerical approximation of $v(t_k,x_n)$.
The Neumann condition $v^{-}_y(t,0)=0$ is modeled with the virtual grid point $y_{-1} = -h$ so that $v^{-}_{-1,k}= v^{-}_{1,k}$. By using the virtual grid point $y_{-1}$ and the interface condition (\ref{redundant-cond}), we also express
\begin{equation}
v^+_{-1,k} = -\frac{2+h}{2-h} v^+_{1,k},
\end{equation}
after which the approximation of $\gamma(t_k)$ is obtained from \eqref{xi-dot} as follows:
\begin{equation}\label{gamma}
\gamma(t_k) = - \frac{(2-h)v^{-}_{1,k}}{h v^{+}_{1,k} + h^2(2-h)}.
\end{equation}

We use the Crank--Nicholson method in order to perform steps in time 
for the evolution system (\ref{v-eqs}). For each equation of the form 
$\frac{dv}{dt} = f(v)$, the Crank--Nicholson method yields:
\begin{align}\label{C-N1}
v_{k+1} -\frac{\tau}{2} f(v_{k+1}) = v_{k} + \frac{\tau}{2} f(v_{k}),  
\end{align}
where $f$ for the first and second equations of system (\ref{v-eqs}) take the form:
\begin{align*}
[f^+]_{n,k} &= \frac{v^{+}_{n+1,k}-v^{+}_{n-1,k}}{2h}+
\frac{v^{+}_{n+1,k}-2v^{+}_{n,k}+v^{+}_{n-1,k}}{h^2}+
\gamma_{k}\frac{v^{-}_{n+1,k}- v^{-}_{n-1,k}}{2h},\\
[f^+]_{n,k} &=   \frac{v^{-}_{n+1,k}-v^{-}_{n-1,k}}{2h}+
\frac{v^{-}_{n+1,k}-2v^{-}_{n,k}+v^{-}_{n-1,k}}{h^2}+
\gamma_{k}\frac{v^{+}_{n+1,k}- v^{+}_{n-1,k}}{2h} + 2\gamma_k e^{-y_n}.
\end{align*}
For simplicity,
we use $\gamma_k$ at the time level $k$ on both sides of equation \eqref{C-N1}.
Thus, in order to advance the solution of \eqref{v-eqs} to the next time level $k+1$, we
have to solve the following algebraic system:
\begin{equation}
L(-\tau) \bold{v}_{k+1} = L(\tau) \bold{v}_{k}+\bold{c}_k
\end{equation}
where $\bold{v}_k$ and $\bold{c}_k$ are the $2N$ vectors with the elements
\begin{equation}
v_{n,k} = v^+_{n,k}, \quad 1 \leq n \leq N, \qquad \mbox{\rm and} \qquad 
v_{n,k} = v^-_{n,k}, \quad N+1 \leq n \leq 2N,
\end{equation}
and
\begin{equation}
c_{n,k} = 0, \quad 1 \leq n \leq N, \qquad \mbox{\rm and} \qquad 
c_{n,k} = 2 \tau \gamma_k e^{-y_n}, \quad N+1 \leq n \leq 2N,
\end{equation}
and $L(\tau)$ is the $(2N\times 2N)$ matrix defined in the block form:
\begin{equation}
L =
\left[
\begin{array}{c|c}
A & B \\
\hline
B & A
\end{array}
\right],
\end{equation}
with $A$ and $B$ are  $(N\times N)$ three-diagonal matrices with the elements:
\begin{align*}
a_{j,j} = 1 - \frac{\tau}{h^2}, \quad
a_{j,j+1} = \frac{\tau}{2} \left( \frac{1}{2h}+\frac{1}{h^2} \right), \quad
a_{j,j-1} =	\frac{\tau}{2} \left( -\frac{1}{2h}+\frac{1}{h^2} \right)
\end{align*}
and
\begin{align*}
b_{j,j} = 0, \quad b_{j,j+1}= \frac{\tau}{4h} \gamma_k, \quad
b_{j,j-1} = -\frac{\tau}{4h} \gamma_k.
\end{align*}
The solution $u(t,y)$ to the boundary-value problem (\ref{Burgers-general}) for $y \in \R$ is recovered from solution $v^{\pm}(t,y)$ to system \eqref{v-eqs} for $y \in \R_+$ by using the transformation (\ref{change var}). 
Finally, we use $y = x - \xi(t)$ 
with $\xi(t) := \int_0^t \gamma(t') dt'$ in order to display $u(t,x)$ versus $x$ on $\R$.

\begin{figure}[htp]
	\centering
	\includegraphics[width=0.48\linewidth]{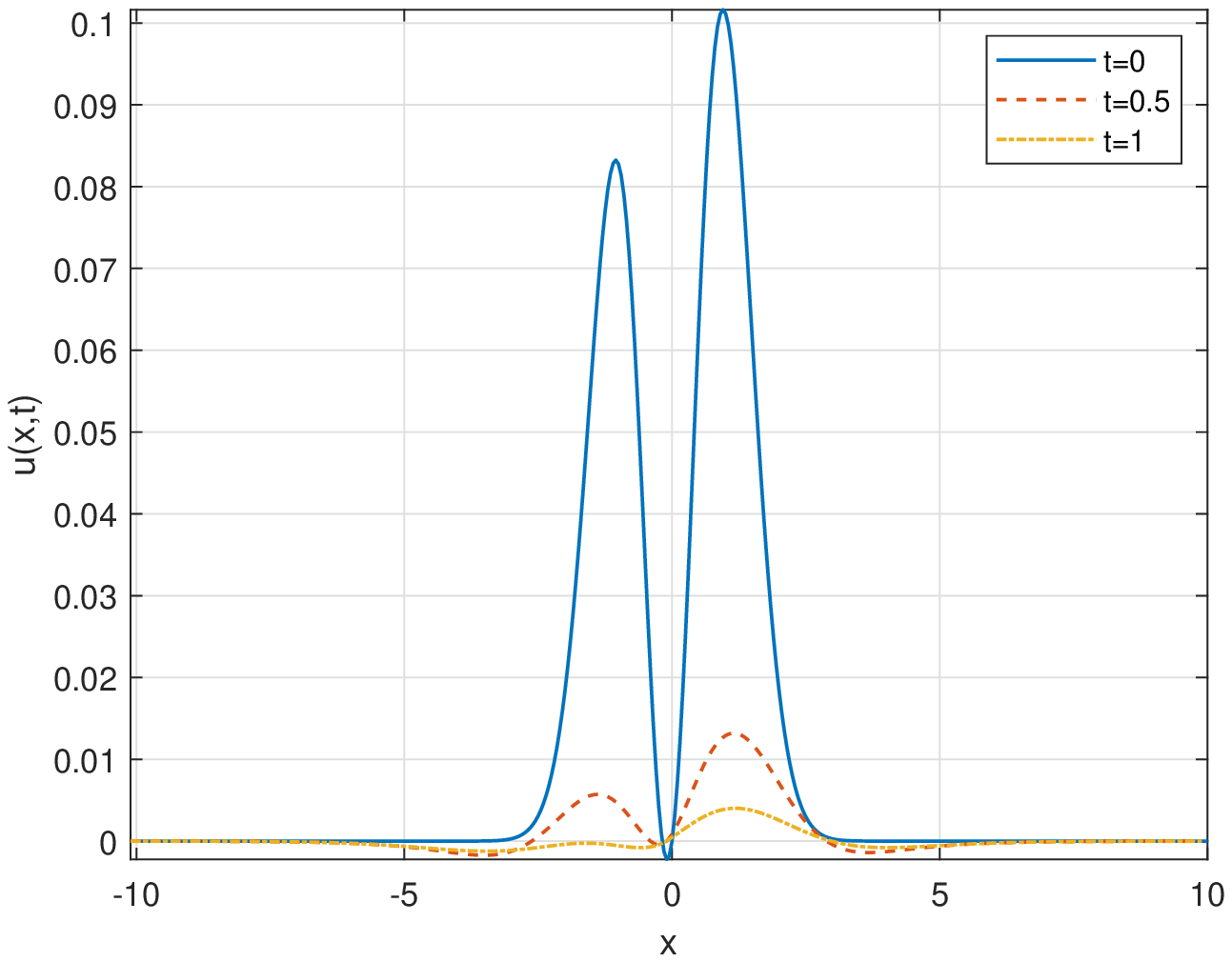}
	\includegraphics[width=0.48\linewidth]{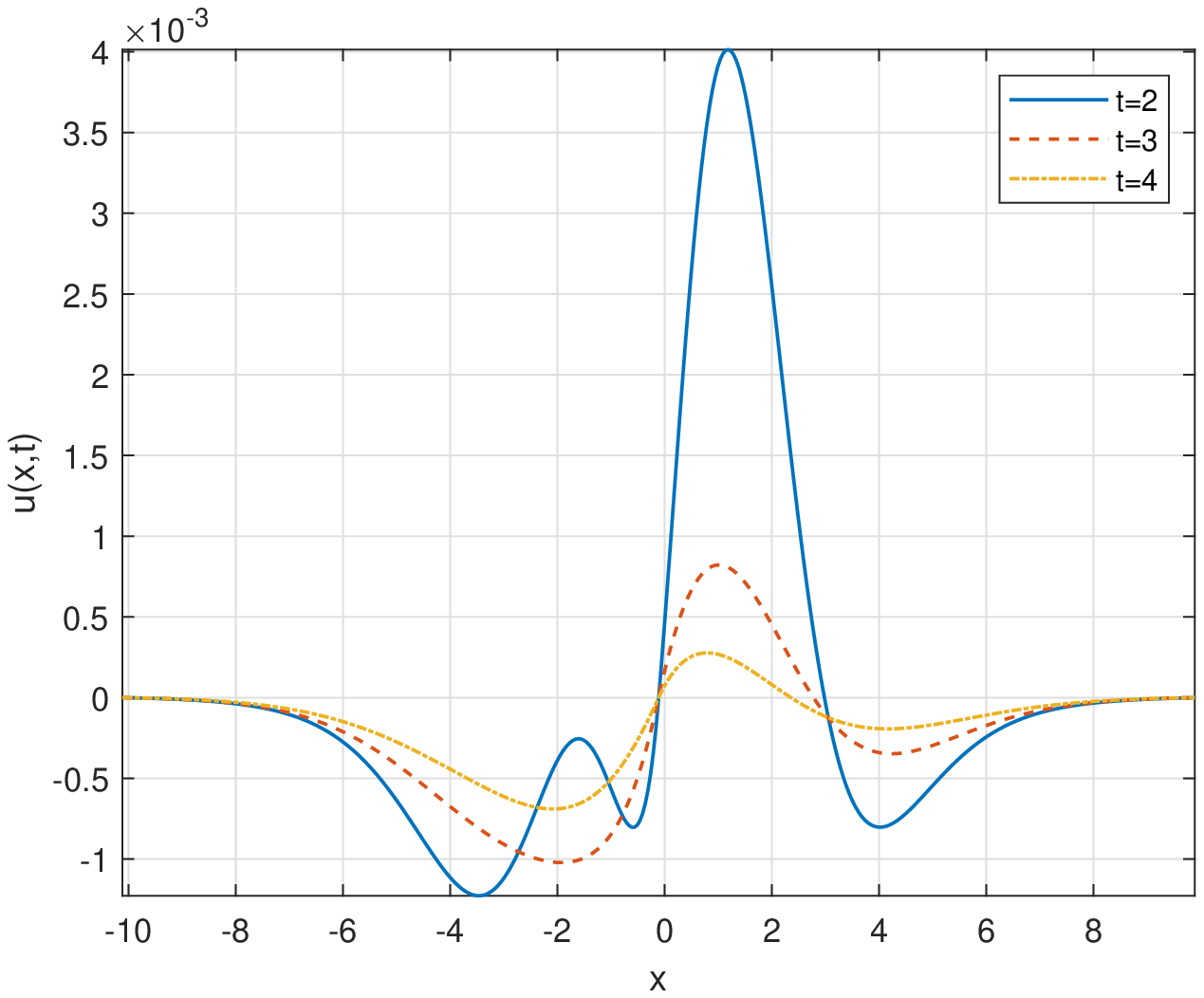} \\
	\includegraphics[width=0.48\linewidth]{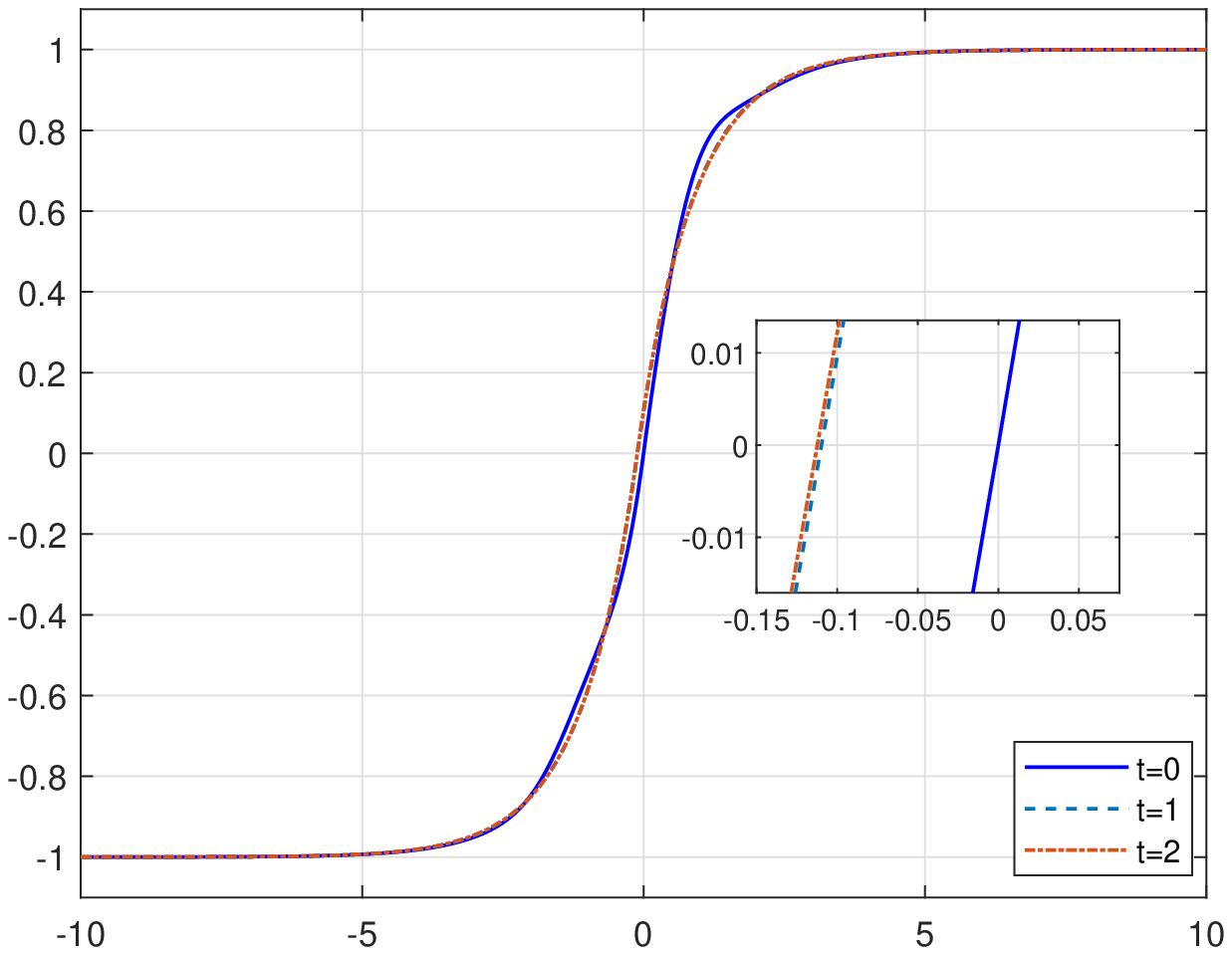}
	\includegraphics[width=0.48\linewidth]{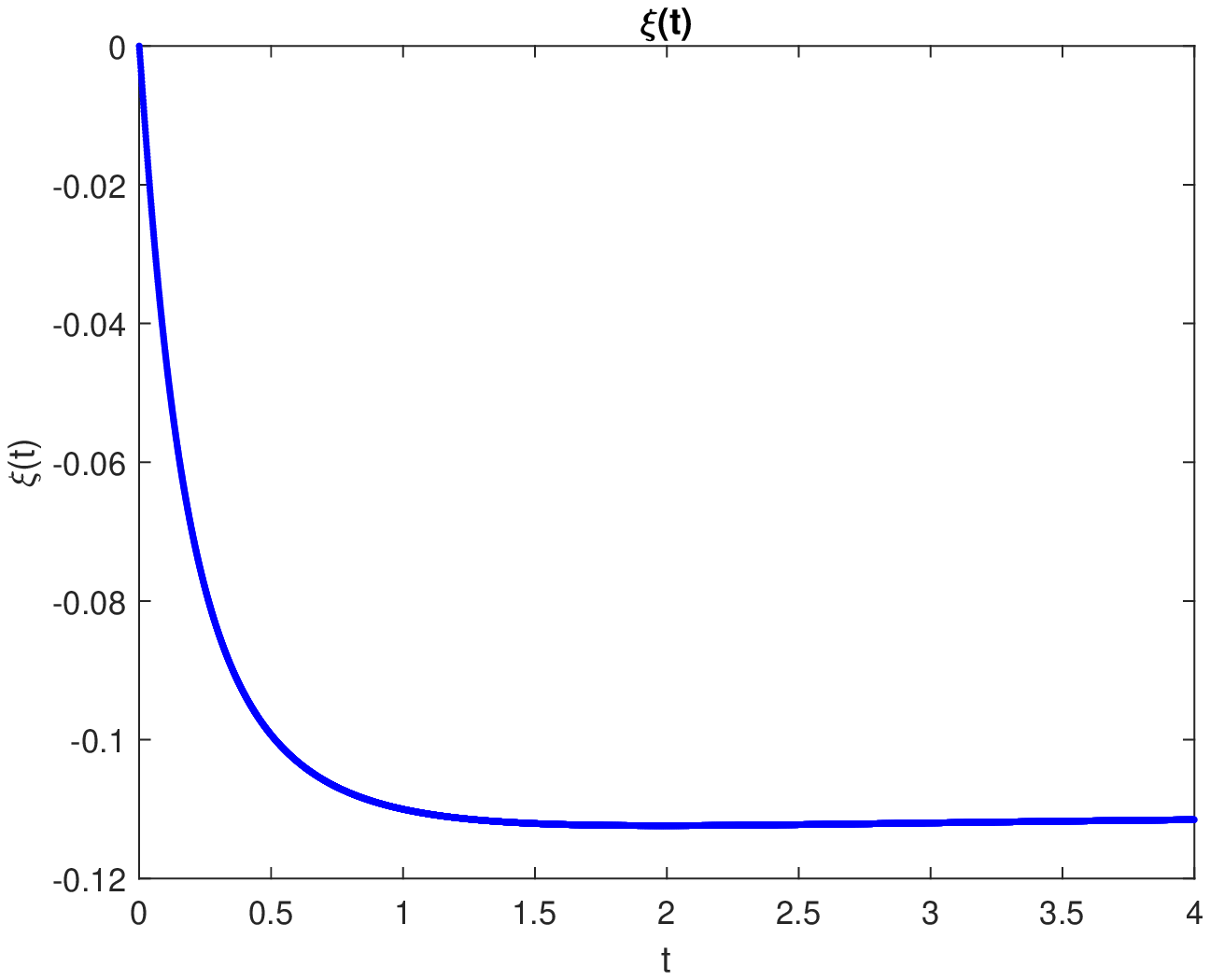}
	\caption{Numerical simulations for the initial conditions \eqref{IC1}.
		Top: plot of $u(t,x)$ versus $x$ for $t = 0, 0.5, 1$ (left) and 
		$t = 2,3,4$ (right). Bottom: plot of $w(t,x)$ versus $x$ for $t = 0,1,2$ (left) and plot of $\gamma(t)$ versus $t$ (right).}
	\label{fig:set1}
\end{figure}

Figure \ref{fig:set1} reports the results of numerical simulations for 
the initial condition with the Gaussian decay:
\begin{align}\label{IC1}
\begin{cases}
v^{+}(0,x)= 0.1(x-0.5x^2)e^{-x^2},\\
v^{-}(0,x)= 0.5x^2e^{-x^2},
\end{cases}
\end{align}
where the coefficients are carefully selected to satisfy 
the boundary conditions in (\ref{v-bc}) and the interface condition 
(\ref{redundant-cond}) at $t = 0$. 

Snapshots of $u(t,x)$ versus $x$ for different values of $t$ (top panels) show that the solution quickly decays to zero in the supremum norm. Although the perturbation $u$ is sign-indefinite, the values of $u$ are smaller compared to the values of $W_0$ in the  viscous shock, hence $w = W_0 + u$ remains positive (negative) to the right (left) of the interface located at $x = \xi(t)$. The snapshots of $w$ are shown on the bottom left panel for $t = 0,1,2$ with the insert showing the profile of $w$ near the interface. 
The bottom right panel shows the position of the interface $\xi$ versus $t$. 
It quickly relaxes to the equilibrium position at $\xi_{\infty} \approx -0.11$.

\begin{figure}[htp]
	\centering
	\includegraphics[width=0.48\linewidth]{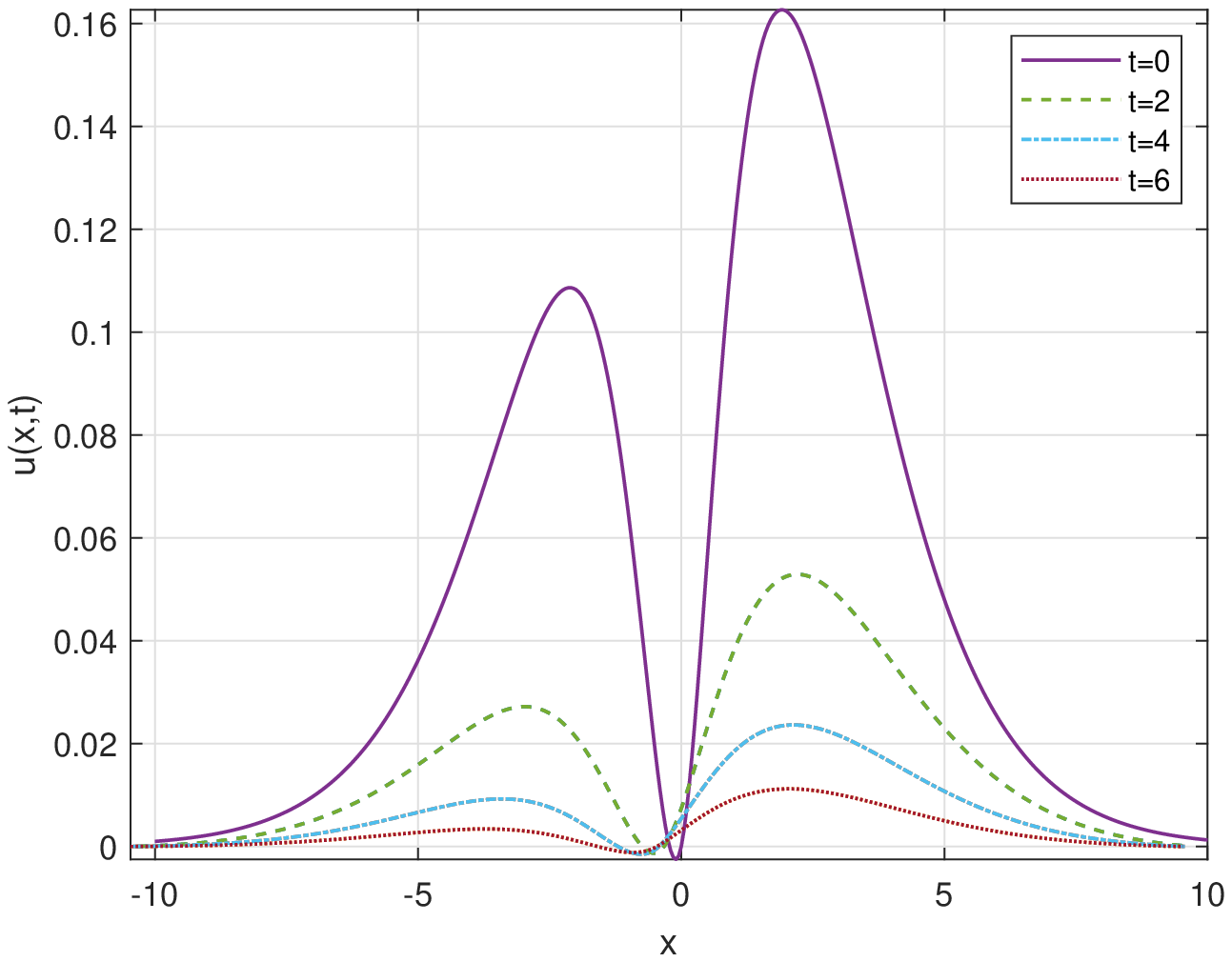}
	\includegraphics[width=0.48\linewidth]{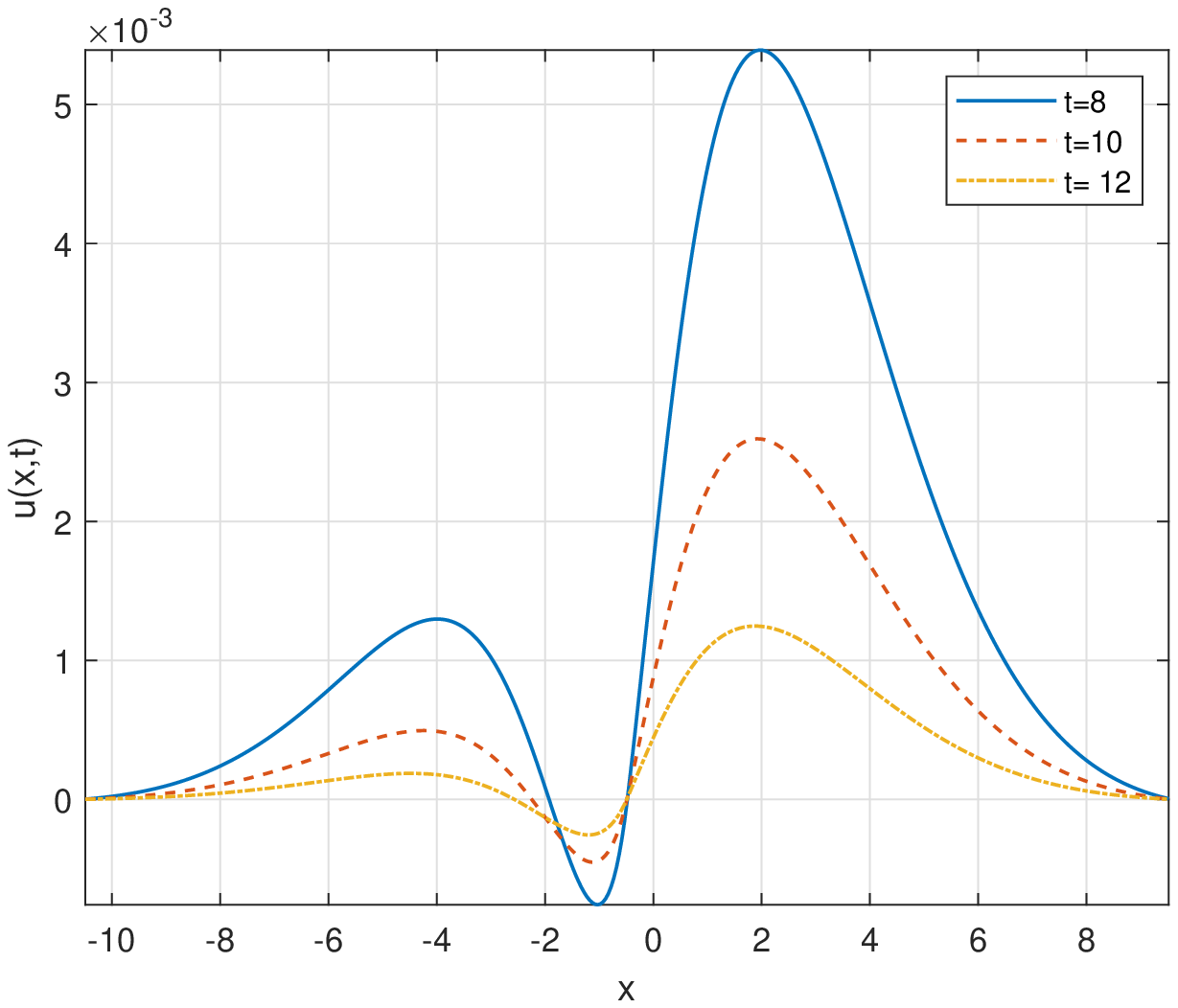} \\
	\includegraphics[width=0.48\linewidth]{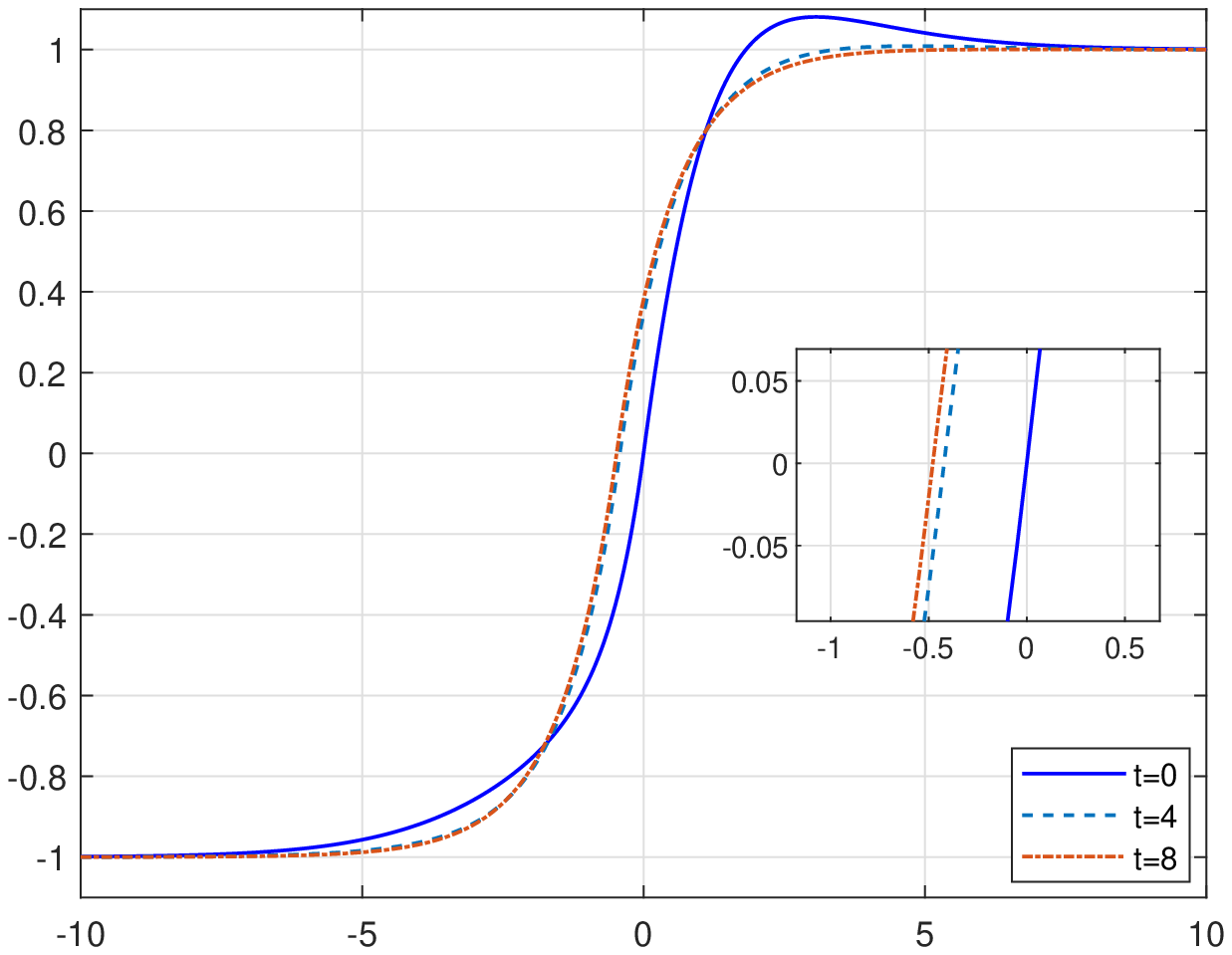}
	\includegraphics[width=0.48\linewidth]{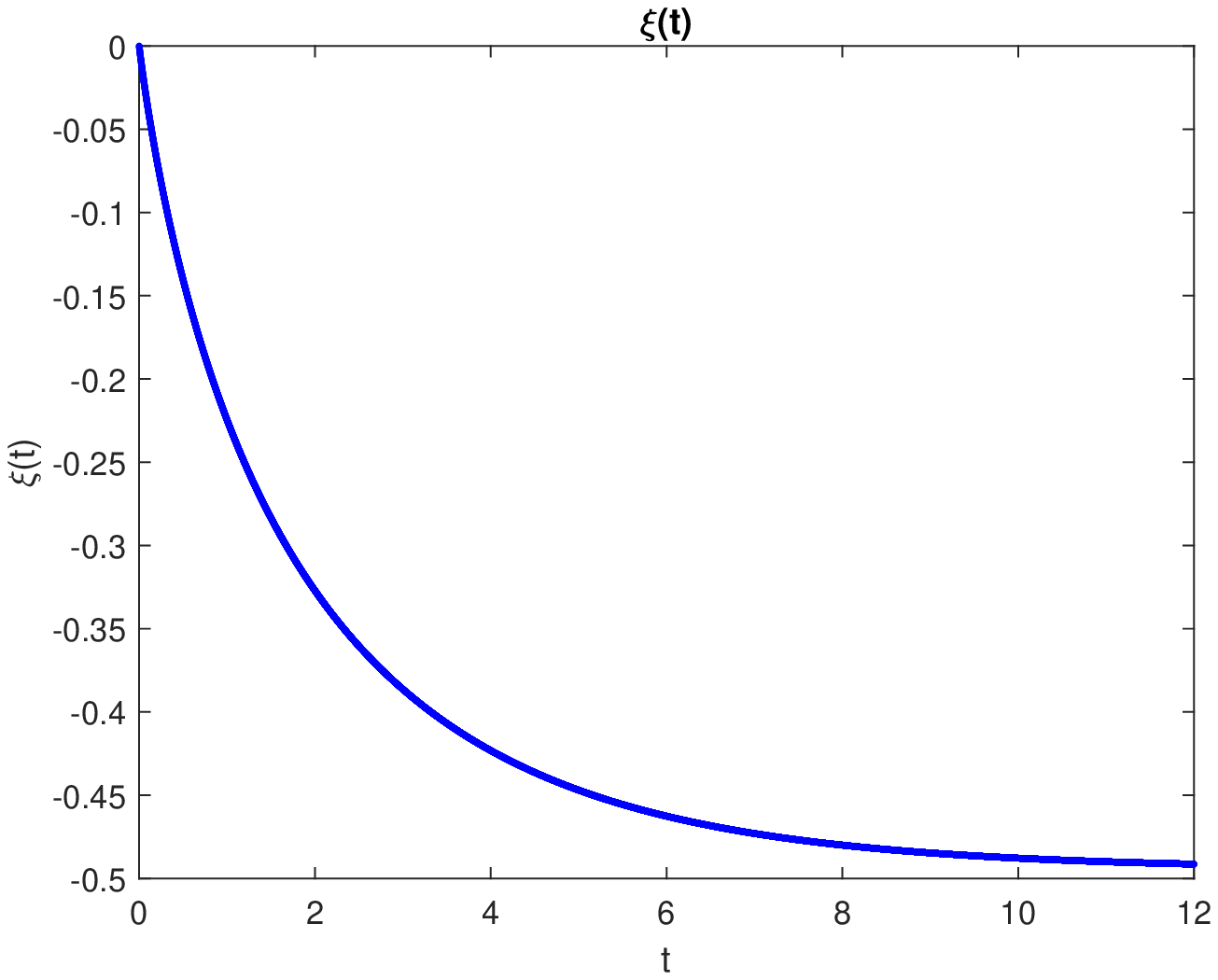}
	\caption{The same as in Figure \ref{fig:set1} but for the initial
		condition \eqref{IC2}.}
	\label{fig:set2}
\end{figure}

Figure \ref{fig:set2} reports similar results for the initial conditions with the exponential decay:
\begin{align}
\begin{cases}\label{IC2}
	v^{+}(x,0)= 0.1(x+0.5x^2)e^{-x},\\
	v^{-}(x,0)= 0.5x^2e^{-x}.
\end{cases}
\end{align}
Dynamics of the perturbation $u$ in time $t$ for the initial data
(\ref{IC2}) resembles the same dynamics as for the initial condition (\ref{IC1}). However, the relaxation time is slower for the exponentially decaying perturbations, hence the time window is extended from $T = 4$ on Figure \ref{fig:set1} to $T=12$ on Figure \ref{fig:set2}. Nevertheless, 
the interface $\xi(t)$ moves to the left and relaxes to some equilibrium 
position $\xi_{\infty} \approx -0.49$.

\section{Conclusion}
\label{section-conclusion}

We have considered the modular Burgers equation, where the advective nonlinearity produces singularities related to the modular functions. 
For the class of viscous shocks with a single interface at the zero value of the modular function, we have proven their asymptotic stability under a general perturbation of sufficient regularity with the spatial exponential decay at infinity. This work may open up new directions of research.

First, it is interesting to consider the existence and nonlinear dynamics of the viscous shocks with multiple interfaces. It is expected that the perturbations at the tails will behave similarly but the dynamics will be complicated by the internal interactions among the interfaces. The periodic waves with an infinite number of interfaces located at the equal distance is another interesting case for further studies, e.g., see \cite{Zumbrun1,Zumbrun2}.

Second, one can wonder if the exponential weight requirement on the initial perturbations can be relaxed or completely removed. It may be relatively easy to replace the exponential weights with the algebraic weights of sufficiently fast decay as done in \cite{B2}. However, we are not able to close the fixed-point arguments for the perturbations to the viscous shocks in $H^2(\mathbb{R}) \cap W^{2,\infty}$, hence new ideas for analysis are needed to remove the weights.

Finally, the Burgers equation with more singular nonlinearity, e.g. given by the logarithmic functions, arises in the applications of granular chains \cite{J20}. It is definitely interesting if the asymptotic stability of viscous shocks can be proven for the logarithmic Burgers equations. Unfortunately, our methods rely on the reductions provided by the modular nonlinearity 
and cannot be extended to the case of logarithmic or other singular nonlinearities.

\vspace{0.25cm}

{\bf Acknowledgements:} Part of this project was completed during the visit of 
D.E. Pelinovsky to LAMIA at Universit\'e des Antilles in December 2019.
He would like to express his gratitude to members of the LAMIA for their hospitality. 
The authors thank S.P. Nuiro for many discussions related to the project.
The research of U. Le and D.E. Pelinovsky is partly supported by the NSERC Discovery grant.

\end{document}